\documentclass{article}

\usepackage[utf8]{inputenc}
\usepackage[english]{babel}
\usepackage[a4paper, margin=0.8in]{geometry}
\usepackage[titletoc]{appendix}
\usepackage{stmaryrd}
\usepackage{amssymb}
\usepackage{amsmath}
\usepackage{graphicx}
\usepackage{enumerate}
\usepackage{enumitem}
\usepackage{multicol}
\usepackage{subcaption}
\usepackage{proof}
\usepackage{amsthm}
\usepackage[normalem]{ulem}
\usepackage{array}
\usepackage{lipsum}

\theoremstyle{plain}
\newtheorem{lemma}{Lemma}
\newtheorem{theorem}{Theorem}

\theoremstyle{definition}
\newtheorem{definition}{Definition}

\theoremstyle{remark}
\newtheorem{example}{Example}

\usepackage{tikz}
\usetikzlibrary{patterns,arrows, topaths, calc, positioning}
\tikzset{>=stealth}
\tikzstyle{node} = [circle, minimum size = 1.4mm, inner sep = 0mm, color=black, fill]
\tikzstyle{hyperedge} = [rectangle, minimum width = 5mm, minimum height = 5mm, draw, inner sep = 0mm]
\tikzstyle{HG} = [align = center]
\tikzstyle{circledge} = [circle, minimum size = 7mm, inner sep = 0mm, color=black, draw]

\pgfdeclarepatternformonly{NELines}{\pgfqpoint{-1pt}{-1pt}}{\pgfqpoint{4pt}{4pt}}{\pgfqpoint{3pt}{3pt}}%
{
	\pgfsetlinewidth{0.4pt}
	\pgfpathmoveto{\pgfqpoint{0pt}{0pt}}
	\pgfpathlineto{\pgfqpoint{1pt}{1pt}}
	\pgfpathmoveto{\pgfqpoint{2pt}{2pt}}
	\pgfpathlineto{\pgfqpoint{3.1pt}{3.1pt}}
	\pgfusepath{stroke}
}

\newcommand{\downsquigarrow}{\mathbin{\rotatebox[origin=c]{-90}{$\rightsquigarrow$}}}
\newcommand{\eqdef}{\mathrel{\mathop:}=}
\newcommand{\neck}{{\mathord{\circlearrowleft}}}
\newcommand{\brac}{{\mathord{\circledcirc}}}
\newcommand{\unneck}{{\mathord{\not\circlearrowleft}}}
\newcommand{\rev}{\mathrm{R}}
\newcommand{\LC}{\mathrm{L}}
\newcommand{\CS}{\mathrm{CS}}
\newcommand{\SG}{\mathop{\mathrm{SG}}}

\title{Cyclic Shift in the Lambek Calculus}
\author{Tikhon Pshenitsyn
	\thanks{The study was supported by RFBR, project number 20-01-00670, by the Theoretical Physics and Mathematics Advancement Foundation ``BASIS'', and by the Interdisciplinary Scientific and Educational School of Moscow University ``Brain, Cognitive Systems, Artificial Intelligence''.}
	\\
	Department of Mathematical Logic and Theory of Algorithms\\Faculty of Mathematics and Mechanics\\
		Lomonosov Moscow State University
		\\GSP-1, Leninskie Gory, Moscow, 119991, Russian Federation
}

\begin{document}
	\maketitle
\begin{abstract}
	We enrich the Lambek calculus with the cyclic shift operation $A^\neck$, which is expected to model the closure operator of formal languages with respect to cyclic shifts. We introduce a Gentzen-style calculus and prove cut elimination. Secondly, we turn to categorial grammars based on this calculus and show that they can generate non-context-free languages; besides, we consider a related calculus where the cyclic shift is a structural rule, and compare recognizing power of these two calculi. Thirdly, we attempt to embed the Lambek calculus with the cyclic shift operation in the hypergraph Lambek calculus. This results in considering a ``bracelet'' operation, which can be defined through the cyclic shift, union, and the reversal operation.
\end{abstract}

\section{Introduction}\label{sec_intr}
	The Lambek calculus is a logical formalism designed to model syntax of natural languages; it was firstly introduced by Joachim Lambek \cite{Lambek58}. In the standard Lambek calculus $\LC$, there is a set of primitive types $Pr$, and types are built from primitive ones using binary connectives $\backslash$, $/$ and $\cdot$. Throughout this paper, we consider the Gentzen-style Lambek calculus (because it is easier to reason about the calculus then); that is, its derivable objects are sequents of the form $A_1\dotsc,A_n\to B$ where $A_i$ and $B$ are all types. Its axioms and rules are the following:
	$$
	\infer[(\mathrm{Ax})]{A\to A}{} 
	\qquad
	\infer[(\backslash\to)]{\Gamma, \Pi, A \backslash B, \Delta \to C}{\Gamma, B, \Delta \to C & \Pi \to A}
	\qquad
	\infer[(\to\backslash)]{\Pi \to A \backslash B}{A, \Pi \to B}
	\qquad
	\infer[(\cdot\to)]{\Gamma, A \cdot B, \Delta \to C}{\Gamma, A, B, \Delta \to C}
	$$
	$$
	\infer[(/\to)]{\Gamma, B / A, \Pi, \Delta \to C}{\Gamma, B, \Delta \to C & \Pi \to A }
	\quad
	\infer[(\to/)]{\Pi \to B / A}{\Pi, A \to B}
	\quad
	\infer[(\to\cdot)]{\Pi, \Psi \to A \cdot B}{\Pi \to A & \Psi \to B}
	\quad
	\infer[(\mathrm{cut})]{\Gamma, \Pi, \Delta \to B}{\Pi \to A & \Gamma, A, \Delta \to B}
	$$
	There is a formal language semantics for types and sequents of the Lambek calculus. Namely, we assign a formal language over some fixed alphabet $\Sigma$ to each primitive type: $w(p)\subseteq \Sigma^\ast$ for each $p\in Pr$; we extend the assignment function $w$ to all types and sequents as follows:
	\begin{multicols}{2}
		\begin{enumerate}
			\item $w(B\backslash A)=\{u\in\Sigma^\ast\mid \forall v\in w(B) \; vu\in w(A)\}$;
			\item $w(A/B)=\{u\in\Sigma^\ast\mid \forall v\in w(B) \; uv\in w(A)\}$;
			\item $w(A\cdot B)=\{uv \mid u \in w(A), v \in w(B)\}$;
			\item $w(A_1,\dots,A_n)=w(A_1)\cdot\dotsc\cdot w(A_n)$ where $w(A) \cdot w(B) = w(A \cdot B)$;
			\item $w(\Pi\to A)$ is true if and only if $w(\Pi)\subseteq w(A)$.
		\end{enumerate}
	\end{multicols}
	It is known that the Lambek calculus is sound and complete w.r.t. this semantics (i.e. $\LC\vdash \Pi \to A$ if and only if $w(\Pi\to A)$ is true for all $w$; see \cite{Pentus95}). Consequently, $\backslash,/,\cdot$ can be interpreted as operations on formal languages, and derivable sequents of the form $A\to B$ as valid formulas over the signature $\backslash,/,\cdot,\subseteq$.
	
	Since 1958 many extensions and variants of $\LC$ have been studied; a number of new operations have been added to it. Some of them have been introduced for linguistic purposes (e.g. $\diamondsuit$ and $\square$ modalities in \cite{Moortgat96} or the $!$ modality studied e.g. in \cite{Kanovich17}). However, several connectives that do not have linguistic applications have also been studied because they correspond to some natural operations on formal languages. Let us look at one of such examples in detail (it will play an important role later). Namely, consider the reversal operation on formal languages: given a language $L$, the language $L^\rev$ is defined as $L^\rev\eqdef\{a_n\dotsc a_1\mid a_1\dotsc a_n\in L\}$. A question arises how to axiomatize this operation along with $\backslash,/,\cdot$. The answer to this question is the Lambek calculus enriched with the reversal operation $^\rev$. In this calculus denoted $\LC^\rev$, a unary operation $A^\rev$ is added; axioms and rules of $\LC^\rev$ include those of $\LC$ and the following ones:
	$$
	\infer[(^\rev\to^\rev)]{A_n^\rev,\dots, A_2^\rev, A_1^\rev \to B^{\rev}}{A_1,A_2,\dots, A_n \to A}
	\qquad
	\infer[(^{\rev\rev}\to)]{\Gamma, A^{\rev\rev},\Delta\to B}{\Gamma, A,\Delta\to B}
	\qquad
	\infer[(\to^{\rev\rev})]{\Pi \to A^{\rev\rev}}{\Pi \to A}
	$$
	It is not a simple question if $\LC^\rev$ is complete w.r.t. the formal language semantics (where we set $w(A^\rev)\eqdef w(A)^\rev$). For $\LC^\rev$ completeness is proved in \cite{Kuznetsov12}. Nevertheless, even if completeness had not been established yet, it would have been interesting to study such a calculus from the syntactic point of view since it is a plausible candidate for the proposed semantics of $^\rev$.
	
	In this paper, we are going to consider another operation on formal languages called a \emph{cyclic shift}. Given a formal language $L$, $L^\neck$ denotes the language of cyclic shifts of words of $L$: $L^\neck\eqdef\{vu\mid uv\in L\}$. It has also been studied in the context of formal languages: e.g. in \cite[pp.143-144]{Hopcroft79} it is proved that context-free languages are closed under the cyclic shift, and in \cite{Jiraskova08, Maslov70} complexity of the cyclic shift is investigated. We aim to develop a calculus describing $^\neck$ and its cooperation with $\backslash$, $/$ and $\cdot$. In Section \ref{sec_Lneck}, we introduce the Lambek calculus with the cyclic shift $\LC^\neck$, and prove the cut elimination theorem; there we also discuss completeness-related issues. In Section \ref{sec_grammars}, we study recognizing power of categorial grammars based on $\LC^\neck$; it appears that they can generate non-context-free languages. In Section \ref{sec_Lbrac}, we attempt to embed $\LC^\neck$ in a generalization of the Lambek calculus to hypergraphs, which is defined and studied in \cite{Pshenitsyn21_2}; this leads us to another connective, which we call the bracelet operation and denote $A^\brac$; from the point of view of the formal language semantics, it is expected to correspond to the following operation: $L^\brac\eqdef L^\neck\cup (L^\rev)^\neck$. In Section \ref{sec_concl}, we summarize our observations and conclude.
\section{Lambek Calculus with the Cyclic Shift}\label{sec_Lneck}
	Before we start discussing the issue of modelling the cyclic shift in $\LC$, let us give some preliminary definitions and denotations. $\mathbb{N}$ contains $0$. $\Sigma^\ast$ is the set of all strings over the alphabet $\Sigma$ (including the empty string $\varepsilon$). $\Sigma^\circledast$ is the set of all strings over $\Sigma$ consisting of distinct symbols. The length $|w|$ of a word $w$ is the number of symbols in $w$. Each function $f:\Sigma\to\Delta$ can be naturally extented to a function $f:\Sigma^*\to\Delta^*$ ($f(\sigma_1\dots\sigma_k)=f(\sigma_1)\dots f(\sigma_k)$).
	\\
	Small Latin letters $p,q,r,\dotsc$ represent primitive types; capital Latin letters $A,B,C,\dotsc$ usually represent types; capital Greek letters $\Gamma,\Delta,\dotsc$ represent sequences of types (and $\Pi,\Psi,\Xi$ denote nonempty sequences). The notation $\mathcal{K}\vdash \Pi \to A$ means that the sequent $\Pi \to A$ is derivable in the calculus $\mathcal{K}$.
	
	Now, we are ready to propose an axiomatization for the $^\neck$ operation. Namely, the following inference rules are introduced:
	$$
	\infer[(\to^\neck)]{\Pi\to A^\neck}{\Pi\to A}
	\qquad
	\infer[(^\neck\to)]{B^\neck\to A^\neck}{B\to A^\neck}
	\qquad
	\infer[(\neck)]{\Psi,\Pi \to A^\neck}{\Pi,\Psi\to A^\neck}
	$$
	The Lambek calculus enriched with the $^\neck$ connective and with these three rules is denoted by $\LC^\neck$; its set of types is denoted $Tp^\neck$. It can be easily checked that all the above rules are sound w.r.t. the formal language semantics.
	\begin{example}
		Below some derivable in $\LC^\neck$ sequents are presented:
		\begin{multicols}{2}
			\begin{itemize}
				\item $\LC^\neck\vdash p^{\neck\neck}\leftrightarrow p^\neck$
				\item $\LC^\neck\vdash q\cdot p\to (p\cdot q)^\neck$
				\item $\LC^\neck\vdash (q\cdot p)^\neck \leftrightarrow (p\cdot q)^\neck$
				\item $\LC^\neck\vdash q\backslash p^\neck \leftrightarrow p^\neck/q$
				\item $\LC^\neck\vdash p\to q^\neck/(q^\neck/p)$
			\end{itemize}
		\end{multicols}
	\end{example}
	First of all, we need to establish the cut elimination property.
	\begin{theorem}
		The cut rule is eliminable in $\LC^\neck$.
	\end{theorem}
	\begin{proof}
		The proof is organized similarly to that presented in \cite{Lambek58} (or in \cite{Takeuti13} for $\mathrm{LK}$); namely, it is done by induction on the total length of $\Pi \to A$ and $\Gamma,A,\Delta\to B$, and it essentially consists of examining several cases according to the last rules applies in the derivation. Let us consider only cases where the cyclic shift operation actively participates, particularly, where it causes a cyclic shift (the remaining cases are considered in \cite{Lambek58}):
		
		\textbf{Case 1.} Let the last rule in $\Gamma,A,\Delta\to B$ be not the one after $A$ appears. Let $B=D^\neck$.
		$$
		\infer[(\mathrm{cut})]{\Gamma, \Pi, \Delta \to B}{\Pi \to A & \infer[(\to^\neck)]{\Gamma, A, \Delta \to D^\neck}{\Gamma, A, \Delta \to D}}
		\quad\rightsquigarrow\quad
		\infer[(\to^\neck)]{\Gamma, \Pi, \Delta \to D^\neck}{\infer[(\mathrm{cut})]{\Gamma, \Pi, \Delta \to D}{\Pi\to A & \Gamma, A, \Delta \to D}}
		$$
		Similarly, the case when $(\neck)$ is applied instead of $(\to^\neck)$ can be dealt with.
		
		\textbf{Case 2.} Let $A=C^\neck$. We are not going to consider cases when the last rule in $\Pi\to A$ is not the one after $A$ appears, or when it is not $(\neck)$ (because these cases are covered by the proof from \cite{Lambek58}). Besides, if $A$ appears at the last step of the derivation within $\Gamma,A,\Delta\to B$, then $\Gamma\Delta=\varepsilon$, and $B=D^\neck$. Therefore, we can reorganize the proof as follows and apply the induction hypothesis:
		\\
		\textit{Case 2a.}
		$$
		\infer[(\mathrm{cut})]{\Pi \to D^\neck}{\infer[(\to^\neck)]{\Pi \to C^\neck}{\Pi \to C} & \infer[(^\neck\to)]{C^\neck \to D^\neck}{C\to D^\neck}}
		\rightsquigarrow 
		\infer[(\mathrm{cut})]{\Pi \to B}{\Pi \to C & C\to D^\neck}
		$$
		\textit{Case 2b.} Let $\Pi=\Phi,\Psi$.
		$$
		\infer[(\mathrm{cut})]{\Phi,\Psi \to D^\neck}{\infer[(\neck)]{\Phi,\Psi \to C^\neck}{\Psi,\Phi \to C^\neck} & \infer[(^\neck\to)]{C^\neck \to D^\neck}{C\to D^\neck}}
		\rightsquigarrow 
		\infer[(\neck)]{\Phi,\Psi \to D^\neck}{\infer[(\mathrm{cut})]{\Psi,\Phi \to D^\neck}{\Psi,\Phi \to C^\neck & C^\neck \to D^\neck}}
		$$
	\end{proof}
	As usually, this implies that the derivability problem for $\LC^\neck$ is decidable. Moreover, it can be seen that it is in NP because derivability can be justified by a derivation tree, which has polynomial size w.r.t. the size of a sequent (notice that, if the rule $(\neck)$ is applied twice in a row, then we can replace these applications by one application of $(\neck)$). Finally, one infers from the cut elimination theorem that $\LC^\neck$ is a conservative extension of $\LC$: if a sequent without $^\neck$ in its types is derivable in $\LC^\neck$, then it is derivable in $\LC$. This allows us to conclude that the derivability problem for $\LC^\neck$ is NP-complete.
	
	We can also define the Lambek calculus with the cyclic shift in a non-sequent way as is done for $\LC$ in \cite[p.163]{Lambek58}. Namely, we consider only sequents of the form $A \to B$ where $A$ and $B$ are types; the calculus contains axioms $A \to A$, $A \cdot (B \cdot C) \leftrightarrow (A \cdot B) \cdot C$, $C\to C^\neck$, and the following rules:
	$$
	\infer[]{A \cdot B \to C}{A \to C/B}
	\quad
	\infer[]{A \cdot B \to C}{B \to A \backslash C}
	\quad
	\infer[]{A \to C/B}{A\cdot B \to C}
	\quad	
	\infer[]{B \to A \backslash C}{A\cdot B \to C}
	\quad
	\infer[]{A \to C}{A \to B & B \to C}
	\quad
	\infer[]{A^\neck \to C^\neck}{A \to C^\neck}
	\quad
	\infer[]{B\cdot A \to C^\neck}{A\cdot B \to C^\neck}
	$$
	Denote the resulting calculus as $\LC^\neck_{\mathrm{H}}$. It is not hard to prove the following
	\begin{theorem}
		$\LC^\neck \vdash A_1, \dotsc, A_n \to B$ if and only if $\LC^\neck_{\mathrm{H}} \vdash A_1\cdot\dotsc\cdot A_n \to B$.
	\end{theorem}
	
	Now, let us turn to the following issue: it can be observed that $^\neck$ occurs in the antecedent of a sequent only when there is a a type with $^\neck$ in the succeedent; the $^\neck$ connective cannot ``do'' anything substantial when being in the antecedent. Can one rid of it somehow in such positions? This question can be answered positively in the sense described below. 
	\begin{definition}\leavevmode
		\begin{enumerate}
			\item Let $A$ be a type and let $B$ be its subtype. We say that $B$ is even in $A$, if one of the following holds:
			\begin{multicols}{2}
				\begin{itemize}
					\item $A=B$;
					\item $A=C^\neck$, and $B$ is even in $C$;
					\item $A=C \cdot D$, and $B$ is even in $C$ or in $D$;
					\item $A=C \backslash D$, and $B$ is odd in $C$ or even in $D$;
					\item $A=D / C$, and $B$ is odd in $C$ or even in $D$.
				\end{itemize}
			\end{multicols}
			\item Let $A$ be a type and let $B$ be its subtype. We say that $B$ is odd in $A$, if one of the following holds:
			\begin{multicols}{2}
				\begin{itemize}
					\item $A=C^\neck$, and $B$ is odd in $C$;
					\item $A=C \cdot D$, and $B$ is odd in $C$ or in $D$;
					\item $A=C \backslash D$, and $B$ is even in $C$ or odd in $D$;
					\item $A=D / C$, and $B$ is even in $C$ or odd in $D$.
				\end{itemize}
			\end{multicols}
		\end{enumerate}
	\end{definition}
	\begin{definition}
		A type is called \emph{even-cyclic} if it does not contain odd subtypes of the form $B^\neck$. A type is called \emph{odd-cyclic} if it does not contain even subtypes of the form $B^\neck$.
	\end{definition}
	\begin{definition}
		The size of a type $A$ is the total number of connectives in $A$. The set of types of $\LC^\neck$ being of size not greater than $N$ is denoted $Tp^\neck_{\le N}$.
	\end{definition}
	\begin{theorem}\label{th_even_calculus}
		For each $N\in \mathbb{N}$ there exist functions $e_N:Tp^\neck_{\le N} \to Tp^\neck$ and $o_N:Tp^\neck_{\le N}\to Tp^\neck$ such that 
		\begin{enumerate}
			\item $e_N(T)$ is even-cyclic, and $o_N(T)$ is odd-cyclic for all $T$ from $Tp^\neck_{\le N}$;
			\item For any sequent $A_1,\dotsc, A_n \to B$ where $A_i,B\in Tp^\neck_{\le N}$ it holds that $\LC^\neck \vdash A_1, \dotsc, A_n \to B$ if and only if $\LC^\neck \vdash o_N(A_1),\dotsc, o_N(A_n) \to e_N(B)$.
		\end{enumerate}
	\end{theorem}
	Therefore, we can rid of $^\neck$ in antecedents of sequents but an elimination procedure will depend on the maximal size of types in a sequent. The proof of this theorem is given in Appendix \ref{app_proof_th_even_calculus}; here we will only introduce $e_N$ and $o_N$ functions. Let $l$, $r$ be two primitive types not occuring in $S$ or in the dictionary of $Gr$. Define $A^\Box \eqdef (l \backslash ((l\cdot A)\cdot r))/r$.
		\begin{definition}
			The function $e_N(A)$ replaces each even subtype of $A$ of the form $B^\neck$ by the type $B^{(\Box\neck)^N}\eqdef B^{\Box\neck\dotsc\Box\neck}$ where $^\Box$ and $^\neck$ are repeated $N$ times; it also replaces each odd subtype of $A$ of the form $B^\neck$ by the type $B^\Box$. The function $o_N(A)$ replaces each odd subtype of $A$ of the form $B^\neck$ by $B^{(\Box\neck)^N}$, and each even subtype of such a form by $B^\Box$.
		\end{definition}
		It is not hard to see that $e_N(A)$ is even-cyclic and that $o_N(A)$ is odd-cyclic. In Appendix \ref{app_proof_th_even_calculus} we prove that for $A_1,\dotsc,A_n,B$ belonging to $Tp^\neck_{\le N}$ $\LC^\neck \vdash A_1, \dotsc, A_n\to B \Leftrightarrow \LC^\neck \vdash o_N(A_1), \dotsc, o_N(A_n)\to e_N(B)$. This completes the proof of Theorem \ref{th_even_calculus}.
	
	Finally, we would like to prove the following
	\begin{theorem}
		Let $\LC^\neck\vdash A \to B$ and let for all $\Gamma$ and $\Delta$, if $\LC^\neck \vdash \Gamma,\Delta \to B$, then $\LC^\neck \vdash \Delta,\Gamma \to B$. Then $\LC^\neck \vdash A^\neck \to B$.
	\end{theorem}
	It implies that, if we consider $(Tp^\neck/\equiv,\to,\cdot)$ (where $A\equiv B$ if $\LC^\neck \vdash A \to B$ and $\LC^\neck \vdash B \to A$) as an ordered semigroup $(M,\le, \cdot)$, then $$a^\neck=\inf\{b \mid (a \le b) \wedge \forall g \forall d (g\cdot d \le b \Rightarrow d \cdot g \le b)\}.$$
	\begin{proof}
		Let us fix $A$ and $B$ and define types $T_n(A)$ (denoted also by $T_n$) as follows:
		\begin{enumerate}
			\item $T_0(A)\eqdef A$;
			\item $T_{n}(A)\eqdef q_n \cdot (T_{n-1}(A)/q_n)$.
		\end{enumerate}
		Here $q_1,q_2,\dotsc$ are primitive types that do not occur neither in $A$ nor in $B$. We claim that $\LC^\neck \vdash T_n \to B$ for all $n$. This is proved by induction on $n$. The case $n=0$ follows from the statement of the theorem. The induction step is proved as follows:
		$$
		\infer[(\cdot\to)]{
			T_{n} \to B
		}{
			\infer[]{
				q_n, (T_{n-1}/q_n) \to B
			}{
				\infer[(/\to)]
				{
					T_{n-1}/q_n,q_n \to B
				}
				{
					T_{n-1} \to B & q_n \to q_n
				}
			}
		}
		$$
		More precisely, we derive the sequent $T_{n-1}/q_n,q_n \to B$ and then use the statement of the theorem, obtaining that $q_n,T_{n-1}/q_n \to B$ is also derivable.
		
		Let $N$ be the size of $B$. Consider a cut-free derivation of $T_N \to B$. There must appear sequents of the form $\Gamma, T_{i-1}/q_i, q_i, \Delta \to B$ and $\Gamma^\prime, q_i,T_{i-1}/q_i, \Delta^\prime \to B^\prime$. Let $\Gamma_i, T_{i-1}/q_i, q_i, \Delta_i \to B_i$ be the first sequent of the first form occurring in the derivation from top to bottom, and let $\Gamma_i^\prime, q_i,T_{i-1}/q_i, \Delta_i^\prime \to B_i^\prime$ be the last sequent of the second form appearing in the derivation from bottom to top. This means that there exist the following steps in the derivation:
		\begin{equation}\label{eq_der}
			\infer[(\cdot\to)]
			{
				\Gamma_i^\prime, T_{i}, \Delta_i^\prime \to B_i^\prime
			}
			{
				\infer[]
				{
					\Gamma_i^\prime, q_i,T_{i-1}/q_i, \Delta_i^\prime \to B_i^\prime
				}
				{
					\infer[]
					{
						\dotsc
					}
					{
						\infer[(/\to)]{
							\Gamma_i, T_{i-1}/q_i, q_i, \Delta_i \to B_i
						}{
							\Gamma_i, T_{i-1}, \Delta_i \to B_i
						}
					}
				}
			}
		\end{equation}
		Given some types $C$ and $D$, if $C$ is a subtype of $D$, let us denote this by $C \preceq D$. It holds that $$B_1 \preceq B_1^\prime \preceq B_2 \preceq B_2^\prime \preceq \dotsc \preceq B_N \preceq B_N^\prime\preceq B.$$
		If for all $i=1,\dotsc, N$ the size of $B_i$ is strictly less than the size of $B_i^\prime$, then the size of $B_N^\prime$ is at least $N+1$, which is greater than the size of $B$. This contradicts $B_N^\prime \preceq B$. Hence there exists such $i=i_0$ that $B_{i_0}=B_{i_0}^\prime$. This implies that the part of the derivation (\ref{eq_der}) for $i=i_0$ denoted by ``$\dotsc$'' does not affect the succeedent of the sequent. 
		
		Notice that $T_{i_0-1}/q_i$ and $q_{i_0}$ changed their order; hence $B_{i_0}=C^\neck$ for some $C$, and the ``$\dotsc$'' part must contain an application of $(\neck)$. What can also be contained in it? There can be rules affecting $\Gamma_{i_0}$ and $\Delta_{i_0}$; however, these sequents of types cannot be completely eliminated, \emph{if at least one of them is not empty}. Therefore, if we write types of each antecedent in this derivation in a circle and consider an arc between $T_{i_0-1}/q_{i_0}$ and $q_{i_0}$ going from the first type to the second one clockwise, there always are some types between them (the only way for them to disappear is by applying one of rules $(\to/)$ or $(\to \backslash)$, but they are not applied in this part of the derivation). However, if we go clockwise from $T_{i_0-1}/q_{i_0}$ to $q_{i_0}$ in the sequence $\Gamma_{i_0}^\prime, q_{i_0},T_{i_0-1}/q_{i_0}, \Delta_{i_0}^\prime$, then we find no other types between them. The only way for this to be possible is that $\Gamma_{i_0}$ and $\Delta_{i_0}$ are empty from the beginning. This means that the sequent $T_{i_0-1}\to C^\neck$ is derivable.
		
		Now, we return to considering the derivation of $T_N\to B$ and we do the following: 
		\begin{itemize}
			\item we remove each type of the form $q_i$;
			\item we replace each type of the form $T_{i-1}/q_i$ or $T_{i-1}$ by the type $A$ for $i=1,\dotsc,i_0-2$;
			\item we replace each type of the form $T_{i-1}/q_i$ or $T_{i-1}$ by the type $A^\neck$ for $i=i_0-1,\dotsc,N$;
			\item we add the following rule application to the place of the original derivation where $T_{i_0-1}\to C^\neck$ appears:
			$$
			\infer[(\to ^\neck)]{A^\neck \to C^\neck}{A \to C^\neck}
			$$
		\end{itemize}
		It is not hard to check that the new derivation is correct. Finally, observe that the resulting sequent is $A^\neck \to B$.
	\end{proof}
	
	
	\section{Recognizing Power of Grammars Based on $\LC^\neck$}\label{sec_grammars}
	One of standard issues regarding variants of the Lambek calculus is recognizing power of their categorial grammars. Let us recall the main definitions.
	\begin{definition}\label{LG}
		Given a calculus $\mathcal{K}$ including the set of types $\mathcal{T}$, a $\mathcal{K}$-grammar is a triple $\langle\Sigma, T, \triangleright\rangle$ where $\Sigma$ is a finite alphabet, $T\in \mathcal{T}$ is a distinguished type, and  $\triangleright\subseteq\Sigma\!\times\! \mathcal{T}$ is a finite binary relation. The set $\{T\in \mathcal{T} \mid \exists \, a: a\triangleright T\}$ is called \emph{the toolbox} of this grammar.
	\end{definition}
	\begin{definition}
		The language $L(G)$ recognized by a $\mathcal{K}$-grammar $G=\langle\Sigma, T, \triangleright\rangle$ is the following set of strings: $L(G)\eqdef \lbrace a_{1}\dots a_{n}\in\Sigma^{+}|\ \exists\, T_{1},\dots,T_{n}\in \mathcal{T}:\forall i = 1, \dots, n\  a_{i}\,\triangleright\, T_{i}, \mathcal{K} \vdash T_{1}, \dots, T_{n} \to T \rbrace$.
	\end{definition}
	If $\mathcal{K}=\LC$, then such grammars are called Lambek grammars (since they are based on the Lambek calculus). The most important result regarding them is that they generate the class of context-free languages without the empty word (see \cite{Pentus93}). For $\mathcal{K}=\LC^\rev$ the same statement holds (this is proved in the PhD thesis of Stepan L. Kuznetsov). 
	
	In this section, our objective is to consider $\LC^\neck$-grammars and to study their recognizing power. Since $\LC^\neck$ is a conservative extension of $\LC$, $\LC^\neck$-grammars can recognize all context-free languages (without the empty word). Do they recognize only context-free languages or not? As we mentioned in Section \ref{sec_intr}, context-free languages are closed under the cyclic shift operation, and at the first glance this is an evidence in favor of the claim that $\LC^\neck$-grammars recognize only context-free languages. However, it is not hard to prove the following
	\begin{theorem}\label{th_perm}
		$\LC^\neck$-grammars recognize all permutations of regular languages.
	\end{theorem}
	A permutation of a language $L$ is the language $\textsc{perm}(L)\eqdef \{a_{\sigma(1)}\dotsc a_{\sigma(n)}\mid a_1\dotsc a_n\in L, \sigma\in S_n\}$. It is known that some permutations of regular languages are not context free, e.g. the language $\textsc{perm}(\{abc\}^+)$, and hence $\LC^\neck$-grammars appear to be more powerful than $\LC$-grammars.
	
	We would like to prove Theorem \ref{th_perm} in a bit more general way than just directly presenting construction for a grammar recognizing permutation of a given regular language. Namely, we would like to consider another calculus $\LC+(\CS)$, which is obtained from the Lambek calculus by adding the following structural rule:
	$$
	\infer[(\CS)]{\Psi,\Pi\to A}{\Pi,\Psi\to A}
	$$
	This rule is close to $(\neck)$ but the difference is that cyclic shift in the latter rule is performed in a controlled way (only if a succeedent is of the form $A^\neck$). Now, we aim to prove two facts.
	\begin{theorem}\label{th_cs_to_neck}
		$\LC+(\CS)$ can be embedded in $\LC^\neck$: there are functions $\mathcal{A}$ and $\mathcal{S}$ from types of the former calculus to those of the latter such that $\LC+(\CS) \vdash \Pi \to A$ if and only if $\LC^\neck \vdash \mathcal{A}(\Pi) \to \mathcal{S}(A)$.
	\end{theorem}
	\begin{theorem}\label{th_perm_cs}
		$(\LC+(\CS))$-grammars recognize all permutations of regular languages.
	\end{theorem}
	\begin{proof}[Proof (of Theorem \ref{th_cs_to_neck}).]
		Define $\mathcal{A}$ and $\mathcal{S}$ inductively as follows: 
		\begin{enumerate}
			\item $\mathcal{A}(p)=p$, $\mathcal{S}(p)=p^\neck$;
			\item $\mathcal{A}(B\backslash A)=\mathcal{S}(B) \backslash \mathcal{A}(A)$, $\mathcal{S}(B\backslash A)=(\mathcal{A}(B)\backslash \mathcal{S}(A))^\neck$;
			\item $\mathcal{A}(A/B)=\mathcal{A}(A)/\mathcal{S}(B)$, $\mathcal{S}(A/B)=(\mathcal{S}(A)/\mathcal{A}(B))^\neck$;
			\item $\mathcal{A}(A\cdot B)=\mathcal{A}(A)\cdot \mathcal{A}(B)$, $\mathcal{S}(A\cdot B)=(\mathcal{S}(A)\cdot \mathcal{S}(B))^\neck$.
		\end{enumerate}
		Given a derivation of the sequent $\LC+(\CS)\vdash \Pi \to A$, we remodel it into a derivation of $\mathcal{A}(\Pi) \to \mathcal{S}(A)$ as follows: each time when a new type appears in a succedent (consider the derivation from top to bottom), we apply $(\to^\neck)$ and put $^\neck$ on it; each application of $(\CS)$ then can be replaced by $(\neck)$. The remaining rules are not changed. This yields a correct derivation in $\LC^\neck$.
		Conversely, if a derivation of $\LC^\neck \vdash \mathcal{A}(\Pi) \to \mathcal{S}(A)$ is given, then one can just remove all $^\neck$ operations and $(\to^\neck)$ rules, and obtain a correct derivation in $\LC+(\CS)$.
	\end{proof}
	\begin{proof}[Proof (of Theorem \ref{th_perm_cs}).]
		Given a regular language $L$, let $Gr = \langle \Sigma, s, \triangleright \rangle$ be an $\LC$-grammar recognizing it such that types of its dictionary are either of the form $p/q$ or of the form $p$ for $p$, $q$ being primitive (such a grammar can be obtained from a right-linear grammar using a standard construction from \cite{Bar-Hillel60}). Note that $s$ is primitive. Let us consider the same grammar $\langle \Sigma, S, \triangleright \rangle$ as an $(\LC+\CS)$-grammar; denote the language recognized by it as $L^\prime$. We argue that $L^\prime=\textsc{perm}(L)$. 
		
		Firstly, we show that $\textsc{perm}(L) \subseteq L^\prime$. Let us prove the following proposition: $w \in \textsc{perm}(L(Gr_q))$ implies $w \in L^\prime(Gr_q)$ where $Gr_q=\langle \Sigma, q, \triangleright \rangle$, and $L^\prime(Gr_q)$ is the language generated by $Gr_q$ considered as an $(\LC+\CS)$-grammar. It is proved by induction on the length of $w$. The case $|w|=1$ is clear. Now, let $w=a_1\dotsc a_{n}\in \textsc{perm}(L(Gr_q))$; that is, there are types $T_i$ such that $a_{\sigma(i)} \triangleright T_{i}$, and $\LC \vdash T_1, \dotsc, T_n \to q$. Then $T_1=q/r$, and $\LC \vdash T_2, \dotsc, T_n \to r$. Hence, $a_{1}\dotsc a_{\sigma(1)-1}a_{\sigma(1)+1}\dotsc a_n \in \textsc{perm}(L(Gr_r))$, and, by the induction hypothesis, it belongs to $L^\prime(Gr_r)$; that is, $a_i \triangleright T_i^\prime$ for $i\ne \sigma(1)$, and $\LC+\CS \vdash T_1^\prime,\dotsc, T_{\sigma(i)-1}^\prime T_{\sigma(i)+1}^\prime \dotsc T_{n}^\prime \to r$. Finally, consider the following derivation:
		$$
		\infer[(\CS)]{
			T_1^\prime,\dotsc, T_{\sigma(i)-1}^\prime, q/r, T_{\sigma(i)+1}^\prime \dotsc T_{n}^\prime \to q
		}
		{
			\infer[(/\to)]
			{
				q/r, T_{\sigma(i)+1}^\prime \dotsc T_{n}^\prime, T_1^\prime,\dotsc, T_{\sigma(i)-1}^\prime \to q
			}
			{
				q \to q
				&
				\infer[(\CS)]{
					T_{\sigma(i)+1}^\prime \dotsc T_{n}^\prime, T_1^\prime,\dotsc, T_{\sigma(i)-1}^\prime \to r
				}
				{
					T_1^\prime,\dotsc, T_{\sigma(i)-1}^\prime, T_{\sigma(i)+1}^\prime \dotsc T_{n}^\prime \to r
				}
			}
		}
		$$
		This justifies that $w$ belongs to $L^\prime(Gr_q)$. The corollary is that $\textsc{perm}(L)=\textsc{perm}(L(Gr_s))\subseteq L^\prime(Gr_s)=L^\prime$.
		
		Secondly, let us show that $L^\prime \subseteq \textsc{perm}(L)$. Let $a_1\dotsc a_n$ belong to $L^\prime$; then there are types $T_1,\dotsc,T_n$ such that $a_i\triangleright T_i$, and $\LC+\CS \vdash T_1,\dotsc,T_n \to s$. There is exactly one $i\in\{1,\dotsc,n\}$ such that $T_i$ is a primitive type. The sequent $T_{i+1},\dotsc,T_n,T_1,\dotsc,T_i\to s$ is also derivable; without loss of generality we can denote its antecedent as $p_1/q_1,\dotsc, p_{n-1}/q_{n-1},p_n$. 
		\begin{lemma}\label{lemma_perm_simple}
			If $\LC+\CS \vdash p_1/q_1,\dotsc, p_{n-1}/q_{n-1},p_n \to s$ for some primitive types $p_i$, $q_i$ and $s$, then there exists a permutation $\sigma\in S_{n-1}$ such that $q_{\sigma(i)}=p_{\sigma(i+1)}$ for $i=1,\dotsc,n-2$, and $q_{\sigma(n-1)}=p_n$, $p_{\sigma(1)}=s$.
		\end{lemma}
		This lemma is proved by a straightforward induction; see Appendix \ref{app_proof_lemma_perm_simple}. It yields that we can permute types in $T_1,\dotsc,T_n\to s$ in such a way that the resulting sequent $T_{\sigma(1)},\dotsc,T_{\sigma(n)}\to s$ will be derivable in the standard Lambek calculus $\LC$; this means that $a_{\sigma(1)}\dotsc a_{\sigma(n)}$ belongs to $L$.
	\end{proof}
	\begin{proof}[Proof (of Theorem \ref{th_perm})]
		It suffices to take an $(\LC+(\CS))$-grammar $\langle \Sigma, s, \triangleright\rangle$ from Theorem \ref{th_perm_cs} and transform it into the grammar $\langle \Sigma, \mathcal{S}(s), \triangleright^\prime \rangle$ where $a\triangleright T$ if and only if $a \triangleright^\prime \mathcal{A}(T)$.
	\end{proof}
	\begin{example}
		The Lambek grammar $\langle\{a,b,c\},s,\triangleright\rangle$ where $a\triangleright s/q/p$, $a\triangleright s/s/q/p$, $b\triangleright p$, $c\triangleright q$ recognizes the language $\{abc\}^+$. Then the grammar $\langle\{a,b,c\},s^\neck,\triangleright^\prime\rangle$ where $a\triangleright^\prime s/q^\neck/p^\neck$, $a\triangleright^\prime s/s^\neck/q^\neck/p^\neck$, $b\triangleright^\prime p$, $c\triangleright^\prime q$ recognizes the language $\textsc{perm}(\{abc\}^+)$.
	\end{example}
	Therefore, we have proved an interesting fact: while context-free languages are closed under cyclic shifts, the cyclic shift operation added in the Lambek calculus places us beyond the context-free bounds. 
	
	Notice also that, if $\mathrm{CFL}$ denotes the set of context-free languages, then not only languages from $\textsc{perm}(\mathrm{CFL})$ can be recognized by $\LC^\neck$-grammars but also those from $\textsc{perm}(\mathrm{CFL})\cdot \textsc{perm}(\mathrm{CFL})$, $\textsc{perm}(\mathrm{CFL})\cdot \textsc{perm}(\mathrm{CFL}) \cdot \textsc{perm}(\mathrm{CFL})$ and so forth. It remains an open question if $\LC^\neck$-grammars can recognize languages that violate the Parikh theorem, in particular, languages over a one-symbol alphabet with nonlinear growth.
	
	Another result is related to Theorem \ref{th_even_calculus}.
	\begin{definition}
		An $\LC^\neck$-grammar $\langle\Sigma, S, \triangleright \rangle$ is called \emph{even-cyclic} if its dictionary contains only odd-cyclic types, and $S$ is even-cyclic.
	\end{definition}
	\begin{theorem}\label{th_even_grammar}
		Each $\LC^\neck$-grammar can be converted into an equivalent even-cyclic grammar.
	\end{theorem}
	This theorem essentially says that we can transform any grammar into an equivalent one in such a way that, when one considers a derivation of sequent made of types of the new grammar, a type of the form $A^\neck$ can occur only in a succeedent of a sequent within such a derivation. 
	\begin{proof}
		An initial grammar $Gr=\langle \Sigma, S, \triangleright \rangle$ is transformed into an even grammar $Gr^\prime=\langle \Sigma, e_N(S), \triangleright^\prime \rangle$ where $a \triangleright^\prime T^\prime$ if and only if $T^\prime = o_N(T)$ such that $a \triangleright T$; $N$ is the maximal size of types in the toolbox of $Gr$ along with $S$. Theorem \ref{th_even_calculus} implies that $L(Gr)=L(Gr^\prime)$.
	\end{proof}
	\begin{definition}
		For each type $A$, $A^\unneck$ is obtained from $A$ by removing all $^\neck$ connectives. Given an $\LC^\neck$-grammar $Gr$, $Gr^\unneck$ is a grammar such that each type $A$ being from the dictionary of $Gr$ or being a distinguished type is replaced by $A^\unneck$.
	\end{definition}
	\begin{definition}
		Given an $\LC$-grammar $Gr$, let us denote the same grammar considered as an $(\LC+\CS)$-grammar by $Gr^{\CS}$.
	\end{definition}
	\begin{theorem}
		For each even grammar $Gr$ it holds that $L(Gr^\unneck)\subseteq L(Gr) \subseteq L((Gr^\unneck)^{\CS})$.
	\end{theorem}
	\begin{proof}[Proof sketch.]
		To prove that $L(Gr^\unneck)\subseteq L(Gr)$ it suffices to consider a derivation of a sequent of the form $A_1^\unneck,\dotsc, A_n^\unneck \to B^\unneck$ and to observe that it can be considered as the derivation of $A_1,\dotsc,A_n \to B$ with that only difference that we need to use $(\to^\neck)$ several times, when necessary. Here we use that the grammar is even because otherwise we would possibly need to use $(^\neck\to)$ in this derivation whereas it is not guaranteed that there will be appropriate conditions to use it.
		
		To prove that $L(Gr) \subseteq L((Gr^\unneck)^{\CS})$ it suffices to notice that, if we have a derivation of $\LC^\neck \vdash A_1,\dotsc,A_n\to B$ where $A_i$ are odd and $B$ is even, and if we remove all $^\neck$ operations from these types, then each application of $(\to^\neck)$ becomes redundant and can be removed, while each application of $(\neck)$ becomes an application of $(\neck)$.
	\end{proof}
	\section{Cyclic Shift in a Hypergraph Lambek Calculus and Bracelet Operation}\label{sec_Lbrac}
	In our work \cite{Pshenitsyn21}, we introduce a generalization of the Lambek calculus to hypergraphs called the hypergraph Lambek calculus $\mathrm{HL}$. This formalism has a cumbersome definition but it does naturally generalize the Lambek calculus to hypergraphs along with its different variants. In our opinion, one of useful features of this calculus is that it enables one to regard seemingly different operations added to the Lambek calculus as parts of a single formalism. $\mathrm{HL}$ underlies fundamental properties of the Lambek calculus that also hold for a number of its variants (e.g. the cut elimination can be proved for it).
	
	The hypergraph Lambek calculus regards strings as oriented string graphs (also called \emph{open chains}, or \emph{open walks}); we will formally introduce this representation later. How can the cyclic shift be represented as an operation on string graphs? Notice that, if one glues the first and the last vertices of a string graph, he/she obtains a cycle; the inverse procedure of ungluing vertices can be done in several ways, and it yields a cyclic shift of the initial string. In this section, we are going to investigate the possibility of representation of $^\neck$ within the hypergraph Lambek calculus using the above idea. To make this work self-contained, in Section \ref{ssec_def_HL} we will succinctly introduce $\mathrm{HL}$, though omitting examples and additional motivation; more details can be seen in \cite{Pshenitsyn21,Pshenitsyn21_2}. 
	
	\subsection{Hypergraph Lambek Calculus}\label{ssec_def_HL}
	Formal definitions of hypergraphs and of hyperedge replacement are given in this section according to the handbook chapter \cite{Drewes97} on HRGs. The remaining definitions are from \cite{Pshenitsyn21}.
	\begin{itemize}
		\item Let us fix the set $C$ of labels along with the function $rank: C\to \mathbb{N}$. \emph{A hypergraph over $C$} is a tuple $G=\langle V_G, E_G, att_G, lab,_G ext_G \rangle$ where $V_G$ is the set of nodes; $E_G$ is the set of hyperedges; $att_G: E_G\to V_G^\ast$ is an attachment function; $lab_G: E_G \to C$ is a labeling function that satisfies the condition $rank(lab_G(e))=|att_G(e)|$ for all $e\in E$; $ext\in V^\circledast$ is the ordered set of \emph{disctint} external nodes. An isomorphism between hypergraphs is defined in a standrard way as a pair of bijections of nodes and hyperedges preserving all functions along with an ordered set of external nodes. Furthermore, we do not distinguish between isomorphic hypergraphs.
		\item $rank_G: E_G\to \mathbb{N}$ is defined as follows: $rank_G(e)\eqdef |att_G(e)|$. Besides, we define $rank(G)\eqdef |ext_G|$.
		\item A hypergraph $H=\langle \{v_i\}_{i=1}^n,\{e_0\},att,lab,v_1\dots v_n\rangle$ where $att(e_0)=v_1\dots v_n$ and $lab(e_0)=a$ is denoted as $a^\bullet$.
		\item A string graph induced by the string $w=a_1\dots a_n$ is a hypergraph $\SG(w):=\\\langle \{v_i\}_{i=0}^n,\{e_i\}_{i=1}^n,att,lab,v_0v_n \rangle$ where $att(e_i)=v_{i-1}v_i$, $lab(e_i)=a_i$ for $i=1,\dotsc,n$.
		\item Replacement of a hyperedge $e$ in a hypergraph $G$ by a hypergraph $H$ can be done, if $rank_G(e)=rank(H)$, as follows: we remove $e$ from $G$, add a copy of $H$ and for all $i$ fuse the $i$-th external node of $H$ with the $i$-th attachment node of $e$. The result of this replacement is denoted as $G[e/H]$.
	\end{itemize}
	Now we start defining the hypergraph Lambek calculus. Firstly, we inductively define the set of types $Tp(\mathrm{HL})$. Note that types will serve as labels on hypergraphs, hence, we need to define $rank$ on them. As in $\mathrm{L}$, in $\mathrm{HL}$ we fix the set of primitive types $Pr$ (but now we also define $rank$ on $Pr$); we assume that for each $n$ there exist infinitely many labels $p\in Pr$ such that $rank(p)=n$. Then, let us introduce two operations of $\mathrm{HL}$ denotes by $\div$ and $\times$:
	\begin{enumerate}
		\item We fix a countable set of labels $\$_n,n\in\mathbb{N}$ and set $rank(\$_n)=n$; let us agree that they do not belong to any other sets considered in the definition of the calculus. 
		\\
		Let $N$ be a type, and let $D$ be a hypergraph such that all labels of its hyperedges, except for one, are already defined types, and one of them equals $\$_d$; let also $rank(N)=rank(D)$. Then $N\div D$ is also a type, and $rank(N\div D)\eqdef d$. The hyperedge of $D$ labeled by $\$_d$ is denoted by $e^\$_D$.
		\item If $M$ is a hypergraph labeled by already defined types, then $\times(M)$ is also a type, and $rank(\times(M))\eqdef rank(M)$.
	\end{enumerate}
	A sequent is a structure of the form $H\to A$ where $H$ is a hypergraph labeled by types, and $A$ is a type.
	
	The hypergraph Lambek calculus $\mathrm{HL}$ deals with hypergraph sequents: a sequent is derivable if it can be obtained from the axiom using rules described below. The only \textbf{axiom} of $\mathrm{HL}$ is $p^\bullet\to p,\quad p\in Pr$. 
	\begin{enumerate}
		\item\label{subsec_div_to} \textbf{Rule $(\div\to)$.} Let $N\div D$ be a type such that $E_D=\{e^\$_D,e_1,\dots,e_k\}$. Let $H\to A$ be a hypergraph sequent and let $e\in E_H$ be labeled by $N$. Then the rule $(\div\to)$ is the following:
		$$
		\infer[(\div\to)]{H[e/D][e^\$_D/ (N\div D)^\bullet][d_1/H_1]\dots[d_k/H_k]\to A}{H\to A & H_1\to lab(d_1) &\dots & H_k\to lab(d_k)}
		$$
		This rule explains how a type with division may appear in the antecedent of a sequent: we replace a hyperedge $e$ by $D$, put a label $N\div D$ instead of $\$_d$ and replace the remaining labels of $D$ by antecedents of corresponding sequents standing in premises. 
		\item \textbf{Rule $(\to\div)$.} 
		$$
		\infer[(\to\div)]{F\to N\div D}{D[e^\$_D/F]\to N}
		$$
		This rule is understood as follows: if there are such hypergraphs $D,F$ and such a type $N$ that in a derivable sequent $H\to N$ $H=D[e_0/F]$, then $F\to N\div D$ is also derivable.
		\item \textbf{Rule $(\times\to)$.} Let $G\to A$ be a hypergraph sequent and let $e\in E_G$ be labeled by $\times(F)$. Then
		$$
		\infer[(\times\to)]{G\to A}{G[e/F]\to A}
		$$
		Intuitively speaking, there is a subgraph $F$ of the antecedent in the premise, and it is ``compressed'' into a single $\times(F)$-labeled hyperedge.
		\item \textbf{Rule $(\to\times)$.} Let $\times(M)$ be a type and let $E_M=\{m_1,\dots,m_l\}$. Then
		$$
		\infer[(\to\times)]{M[m_1/H_1]\dots[m_l/H_l]\to\times(M)}{H_1\to lab(m_1) & \dots & H_l\to lab(m_l)}
		$$
		This means that several sequents can be combined into a single one via a hypergraph structure $M$.
	\end{enumerate}
	
	\subsection{An Attempt to Embed $\LC^\neck$ in $\mathrm{HL}$}
	How can the Lambek calculus be embedded in the hypergraph Lambek calculus? In \cite{Pshenitsyn21}, we present the following translating function:
	\begin{multicols}{2}
		\begin{itemize}
			\item $tr(p)\eqdef p, \quad p\in Pr, type(p)=2$;
			\item $tr(A/B)\eqdef tr(A)\div\SG(\$_2\:tr(B))$;
			\item $tr(B\backslash A)\eqdef tr(A)\div\SG(tr(B)\:\$_2)$;
			\item $tr(A\cdot B)\eqdef \times(\SG(tr(A)\:tr(B)))$.
		\end{itemize}
	\end{multicols}
	A sequent $A_1,\dotsc,A_n \to B$ is translated into the sequent $\SG(tr(A_1),\dotsc,tr(A_n)) \to tr(B)$.
	\begin{example}
		The type $r\backslash (p\cdot q)$ is translated into the type
		$$
		\times\left(\mbox{	
			{\tikz[baseline=.1ex]{
					\node[node,label=left:{\scriptsize $(1)$}] (N1) {};
					\node[node,right=8mm of N1] (N2) {};
					\node[node,right=8mm of N2,label=right:{\scriptsize $(2)$}] (N3) {};
					\draw[->,black] (N1) -- node[above] {$p$} (N2);
					\draw[->,black] (N2) -- node[above] {$q$} (N3);
		}}}\right)\div\left(\mbox{	
			{\tikz[baseline=.1ex]{
					\node[node,label=left:{\scriptsize $(1)$}] (N1) {};
					\node[node,right=8mm of N1] (N2) {};
					\node[node,right=8mm of N2,label=right:{\scriptsize $(2)$}] (N3) {};
					\draw[->,black] (N1) -- node[above] {$r$} (N2);
					\draw[->,black] (N2) -- node[above] {$\$_2$} (N3);
		}}}\right)
		$$
		Here we draw a hypergraph to visualize the internal structure of the type. Black circles represent nodes; arrows represent edges of rank 2 going from the first attachment node to the second one; external nodes are depicted by numbers in brackets.
	\end{example}
	
	It is not hard to observe that the rules of $\mathrm{HL}$ turn into the rules of $\mathrm{L}$ if all involved hypergraphs are string graphs (this is a theorem from \cite{Pshenitsyn21}). We emphasize that in the Lambek calculus there is a restriction that an antecedent must be nonempty; in the hypergraph calculus, this restriction is expressed by the fact that external nodes are distinct (since the empty string corresponds to the string graph $E$ containing one node $v$ and no edges such that $ext_E=vv$; in our framework such hypergraphs are forbidden). 
	
	We are interested in embedding types $A^\neck$ in $\mathrm{HL}$ using some hypergraph construction. We are going to use the following hypergraph:
	$$
	Loop(A)\eqdef 
	\vcenter{\hbox{{\tikz[baseline=.1ex] {
					\node [node] {} edge [in=45,out=-45,loop,label=right:{}] (N1);
					\node[right=2.5mm of N1] {$A$} (N2);
	}}}}
	= \langle\{v_l\},\{e_l\},att,lab,\varepsilon\rangle \mbox{ where } att(e_l)=v_0v_0\mbox{, }lab(e_l)=A
	$$
	Look at the following derivation in $\mathrm{HL}$:
	$$
	\infer[(\to\div)]
	{
		\vcenter{\hbox{{\tikz[baseline=.1ex]{
						\node[node,label=left:{\scriptsize $(1)$}] (N1) {};
						\node[node,right=8mm of N1] (N2) {};
						\node[node,right=8mm of N2] (N3) {};
						\node[node,right=8mm of N3,label=right:{\scriptsize $(2)$}] (N4) {};
						\draw[->,black] (N1) -- node[above] {$q$} (N2);
						\draw[->,black] (N2) -- node[above] {$r$} (N3);
						\draw[->,black] (N3) -- node[above] {$p$} (N4);
		}}}}
		\to 
		\times(Loop(\times(\SG(pqr))))\div Loop(\$_2)
	}
	{
		\infer[(\to\times)]
		{
			\vcenter{\hbox{{\tikz[baseline=.1ex] {
							\node[node] (N1) {};
							\node[node, right=6mm of N1] (N2) {};
							\node[node, above right=5.2mm and 3mm of N1] (N3) {};
							\draw[>=stealth,->,black] (N1) to[bend right = 55] node[below] {$p$} (N2);
							\draw[>=stealth,->,black] (N2) to[bend right = 55] node[above right] {$q$} (N3);
							\draw[>=stealth,->,black] (N3) to[bend right = 55] node[above left] {$r$} (N1);
			}}}}
			\to
			\times(Loop(\times(\SG(pqr))))
		}
		{
			\infer[(\to\times)]
			{
				\vcenter{\hbox{{\tikz[baseline=.1ex]{
								\node[node,label=left:{\scriptsize $(1)$}] (N1) {};
								\node[node,right=8mm of N1] (N2) {};
								\node[node,right=8mm of N2] (N3) {};
								\node[node,right=8mm of N3,label=right:{\scriptsize $(2)$}] (N4) {};
								\draw[->,black] (N1) -- node[above] {$p$} (N2);
								\draw[->,black] (N2) -- node[above] {$q$} (N3);
								\draw[->,black] (N3) -- node[above] {$r$} (N4);
				}}}}
				\to 
				\times(\SG(pqr))
			}
			{
				p^\bullet \to p & q^\bullet \to q & r^\bullet \to r
			}
		}
	}
	$$
	The second application of $(\to\times)$ is performed as follows: we replace the only hyperedge in $Loop(A)$ by the antecedent of the second sequent, namely, by $\SG(pqr)$; this results in gluing external nodes of this string graph. The resulting cycle does not have external edges, therefore, it contains no information about where is the beginning and the end of the former string graph. The rule $(\to\div)$ works here as an inverse procedure: in the resulting sequent such a hypergraph stands that, if we replace the only hyperedge in $Loop(\$_2)$ by it, the result will coincide with the hypergraph standing in the premise. Therefore, the derivation is correct. Hence we have a plausible candidate to play the role of $^\neck$:
	\begin{definition}
		$\neck_{\mathrm{H}}(A)\eqdef \times(Loop(A)) \div Loop(\$_2)$.
	\end{definition}
	We will check later that the following rules are admissible in $\mathrm{HL}$:
	$$
	\infer[]{\SG(\Pi) \to \neck_{\mathrm{H}}(A)}{\SG(\Pi) \to A}
	\qquad
	\infer[]{\SG(\Psi,\Pi) \to \neck_{\mathrm{H}}(A)}{\SG(\Pi,\Psi) \to \neck_{\mathrm{H}}(A)}
	\qquad
	\infer[]{(\neck_{\mathrm{H}}(B))^\bullet \to \neck_{\mathrm{H}}(A)}{B^\bullet \to \neck_{\mathrm{H}}(A)}
	$$
	Therefore, if we define $tr(A^\neck)$ as $\neck_{\mathrm{H}}(tr(A))$, then each derivable in $\LC^\neck$ sequent is converted into a derivable in $\mathrm{HL}$ sequent. However, the converse turns out to be false. 
	\begin{example}\label{ex_two_bad_cases}
		There can be two kinds of undesirable cases in the hypergraph Lambek calculus:
	$$
	\infer[(\to\div)]
	{
		\vcenter{\hbox{{\tikz[baseline=.1ex] {
						\node[node,label=right:{\scriptsize $(2)$}] (N1) {};
						\node[node, below right=3mm and 5.2mm of N1] (N2) {};
						\node[node, above right= 3mm and 5.2mm of N1] (N3) {};
						\node[node,left=2mm of N1,label=left:{\scriptsize $(1)$}] (N4) {};
						\draw[>=stealth,->,black] (N1) to[bend right = 55] node[below] {$p$} (N2);
						\draw[>=stealth,->,black] (N2) to[bend right = 55] node[above right] {$q$} (N3);
						\draw[>=stealth,->,black] (N3) to[bend right = 55] node[above left] {$r$} (N1);
		}}}}
		\to 
		\neck_{\mathrm{H}}(\times(\SG(pqr)))
	}
	{
		\vcenter{\hbox{{\tikz[baseline=.1ex] {
							\node[node] (N1) {};
							\node[node, right=6mm of N1] (N2) {};
							\node[node, above right=5.2mm and 3mm of N1] (N3) {};
							\draw[>=stealth,->,black] (N1) to[bend right = 55] node[below] {$p$} (N2);
							\draw[>=stealth,->,black] (N2) to[bend right = 55] node[above right] {$q$} (N3);
							\draw[>=stealth,->,black] (N3) to[bend right = 55] node[above left] {$r$} (N1);
			}}}}
			\to
			\times(Loop(\times(\SG(pqr))))
	}
	\quad
	\infer[(\to\div)]
	{
		\vcenter{\hbox{{\tikz[baseline=.1ex]{
						\node[node,label=left:{\scriptsize $(1)$}] (N1) {};
						\node[node,right=6.6mm of N1] (N2) {};
						\node[node,right=6.6mm of N2] (N3) {};
						\node[node,right=6.6mm of N3,label=right:{\scriptsize $(2)$}] (N4) {};
						\draw[<-,black] (N1) -- node[above] {$r$} (N2);
						\draw[<-,black] (N2) -- node[above] {$q$} (N3);
						\draw[<-,black] (N3) -- node[above] {$p$} (N4);
		}}}}
		\to 
		\neck_{\mathrm{H}}(\times(\SG(pqr)))
	}
	{
		\vcenter{\hbox{{\tikz[baseline=.1ex] {
						\node[node] (N1) {};
						\node[node, right=6mm of N1] (N2) {};
						\node[node, above right=5.2mm and 3mm of N1] (N3) {};
						\draw[>=stealth,->,black] (N1) to[bend right = 55] node[below] {$p$} (N2);
						\draw[>=stealth,->,black] (N2) to[bend right = 55] node[above right] {$q$} (N3);
						\draw[>=stealth,->,black] (N3) to[bend right = 55] node[above left] {$r$} (N1);
		}}}}
		\to
		\times(Loop(\times(\SG(pqr))))
	}
	$$
	Both cases seem harmless since they lead to hypergraphs that are not string graphs. In the first example, the cycle does not split and does not transform into a string back; an additional isolated node appears. We suspect that this case indeed does not lead to derivable sequents corresponding to underivable sequents in $\LC^\neck$ (but we failed to prove it; see \ref{ssec_brac}). However, look at the following derivation:
	$$
	\infer[]{
		\vcenter{\hbox{{\tikz[baseline=.1ex]{
						\node[node,label=left:{\scriptsize $(1)$}] (N1) {};
						\node[node,right=6.6mm of N1] (N2) {};
						\node[node,right=6.6mm of N2] (N3) {};
						\node[node,right=6.6mm of N3,label=right:{\scriptsize $(2)$}] (N4) {};
						\draw[->,black] (N1) -- node[above] {$r$} (N2);
						\draw[->,black] (N2) -- node[above] {$q$} (N3);
						\draw[->,black] (N3) -- node[above] {$p$} (N4);
		}}}}
		\to 
		tr(((p^\neck \cdot q^\neck)\cdot r^\neck)^\neck)
	}{
		\infer[(\to\times)]{
			\vcenter{\hbox{{\tikz[baseline=.1ex]{
							\node[node,label=left:{\scriptsize $(1)$}] (N1) {};
							\node[node,right=6.6mm of N1] (N2) {};
							\node[node,right=6.6mm of N2] (N3) {};
							\node[node,right=6.6mm of N3,label=right:{\scriptsize $(2)$}] (N4) {};
							\draw[<-,black] (N1) -- node[above] {$p$} (N2);
							\draw[<-,black] (N2) -- node[above] {$q$} (N3);
							\draw[<-,black] (N3) -- node[above] {$r$} (N4);
			}}}}
			\to 
			tr((p^\neck \cdot q^\neck)\cdot r^\neck)
		}
		{
			\infer[(\to\times)]{
				\vcenter{\hbox{{\tikz[baseline=.1ex]{
								\node[node,label=left:{\scriptsize $(1)$}] (N2) {};
								\node[node,right=6.6mm of N2] (N3) {};
								\node[node,right=6.6mm of N3,label=right:{\scriptsize $(2)$}] (N4) {};
								\draw[<-,black] (N2) -- node[above] {$p$} (N3);
								\draw[<-,black] (N3) -- node[above] {$q$} (N4);
				}}}}
				\to 
				tr(p^\neck\cdot q^\neck)
			}
			{
				\infer[]{
					\vcenter{\hbox{{\tikz[baseline=.1ex]{
									\node[node,label=left:{\scriptsize $(1)$}] (N1) {};
									\node[node,right=8mm of N1,label=right:{\scriptsize $(2)$}] (N2) {};
									\draw[<-,black] (N1) -- node[above] {$p$} (N2);
					}}}}
					\to 
					tr(p^\neck)
				}
				{
					p^\bullet \to p
				}
				&
				\infer[]{
					\vcenter{\hbox{{\tikz[baseline=.1ex]{
									\node[node,label=left:{\scriptsize $(1)$}] (N1) {};
									\node[node,right=8mm of N1,label=right:{\scriptsize $(2)$}] (N2) {};
									\draw[<-,black] (N1) -- node[above] {$q$} (N2);
					}}}}
					\to 
					tr(q^\neck)
				}
				{
					q^\bullet \to q
				}
			}
			&
			\infer[]{
				\vcenter{\hbox{{\tikz[baseline=.1ex]{
								\node[node,label=left:{\scriptsize $(1)$}] (N1) {};
								\node[node,right=8mm of N1,label=right:{\scriptsize $(2)$}] (N2) {};
								\draw[<-,black] (N1) -- node[above] {$r$} (N2);
				}}}}
				\to 
				tr(r^\neck)
			}
			{
				r^\bullet \to r
			}
		}
	}
	$$
	Here transitions without markers of rules represent applications of the rules $(\to\times)$ and $(\to\div)$. The problem is that we ``flipped'' string graphs with $p$, $q$, and $r$ once, then we combined them using the product, and then we ``flipped'' the resulting hypergraph again, which lead us to a string graph. However, the sequent $r,q,p\to ((p^\neck\cdot q^\neck)\cdot r^\neck)$ is not derivable in $\LC^\neck$ (this is checked by a direct cut-free proof search).
	\end{example}
	Therefore, $\neck_{\mathrm{H}}(A)$ behaves differently from $A^\neck$ in the hypergraph case: it has an additional ability to flip string graphs. The reason for that is that a cycle graph can be rotated, which yields an isomorphic graph, but also it can be flipped. Similar objects appear in combinatorics; there they are called bracelets (as in \cite{Zelenyuk14}). 
	
	\subsection{Bracelet Operation}\label{ssec_brac}
	Now, we are interested in axiomatizing $\neck_{\mathrm{H}}$ in the standard Lambek calculus without appealing to hypergraphs. It appears that this can be done using the reversal operation (since it enables one to flip, or reverse, strings). 
	\begin{definition}
		The Lambek calculus with the bracelet operation $\LC^\brac$ is the calculus, which types are built using $\backslash$, $/$, $\cdot$, $^\rev$, and the unary operation $^\brac$; it includes the axiom and rules of $\LC^\rev$, and the following ones:
		$$
		\infer[(\to^\brac)]{\Pi\to A^\brac}{\Pi\to A}
		\qquad
		\infer[(^\brac\to)]{B^\brac\to A^\brac}{B\to A^\brac}
		\qquad
		\infer[(\brac)]{\Psi,\Pi \to A^\brac}{\Pi,\Psi\to A^\brac}
		\qquad
		\infer[(\mathrm{Ax}^{\rev\brac})]{A^\rev \to A^\brac}{}
		$$
	\end{definition}
	The distinguishing feature of this calculus is the new axiom $(\mathrm{Ax}^{\rev\brac})$. Note that the cut rule is also included in $\LC^\brac$, and the cut elimination theorem does not hold for it.
	\begin{example}
		The sequent $r,q,p \to (p^\brac \cdot q^\brac)\cdot r^\brac)^\brac$ is derivable in $\LC^\brac$:
		$$
		\infer[(\mathrm{cut})]{
			r,q,p \to (p^\brac \cdot q^\brac)\cdot r^\brac)^\brac
		}{
			\infer[]{
				r,q,p \to (p^\brac \cdot q^\brac)\cdot r^\brac)^\rev
			}
			{
				\infer[(\to\cdot)]{
					p^\rev,q^\rev,r^\rev \to (p^\brac \cdot q^\brac)\cdot r^\brac
				}
				{
					\infer[(\to\cdot)]{
						p^\rev,q^\rev \to p^\brac \cdot q^\brac
					}{
						p^\rev \to p^\brac & q^\rev \to q^\brac
					}
					&
					r^\rev \to r^\brac
				}
			}
			&
			(p^\brac \cdot q^\brac)\cdot r^\brac)^\rev \to (p^\brac \cdot q^\brac)\cdot r^\brac)^\brac
		}
		$$
		A transition without rule markers consists of one application of $(^\rev\to^\rev)$ and of three applications of $(^\rev\to)$. This sequent is not derivable without the cut rule.
	\end{example}
	Now, let us define 
		\begin{itemize}
			\item $tr(A^\brac) \eqdef \neck_{\mathrm{H}}(tr(A))$;
			\item $tr(A^\rev) \eqdef \times\left(\vcenter{\hbox{{\tikz[baseline=.1ex]{
							\node[node,label=below:{\scriptsize $(1)$}] (N1) {};
							\node[node,right=10mm of N1,label=below:{\scriptsize $(2)$}] (N2) {};
							\draw[<-,black] (N1) -- node[above] {$tr(A)$} (N2);
			}}}}
			\right)=\langle \{v_0,v_1\}, \{e_0\}, att,lab,v_1v_0 \rangle$ where $att(e_0)=v_0v_1$, $lab(e_0)=tr(A)$.
		\end{itemize}
	We would expect that $\LC^\brac \vdash A_1,\dotsc, A_n \to B$ if and only if $\mathrm{HL} \vdash \SG(tr(A_1) \dotsc tr(A_n)) \to tr(B)$. However, we failed to prove this because of the following: in a derivation of $\SG(tr(A_1) \dotsc tr(A_n)) \to tr(B)$, there can appear sequents with antecedents of complex structure, which behaviour is hard to describe; the problem arises in derivations similar to that in Example \ref{ex_two_bad_cases} where an isolated node appear. Consideration of specific derivations showed us that a derivation involving ``bad'' antecedents can be made simpler. Hence, we claim that the above statement is true, but proving it remains an open question. Nevertheless, we can avoid these difficulties by ad hoc forbidding isolated nodes. Namely, let us consider a variant of the hypergraph Lambek calculus, in which all hypergraphs must be without isolated nodes (either ones occuring within types or antecedents of sequents), and all the remaining is the same; denote this calculus as $\mathrm{HL}^{WI}$. Then
	\begin{theorem}\label{th_embed}
		$\LC^\brac \vdash A_1,\dotsc, A_n \to B$ if and only if $\mathrm{HL}^{WI} \vdash \SG(tr(A_1) \dotsc tr(A_n)) \to tr(B)$.
	\end{theorem}
	
	The proof is technical but is straightforward. There we use the fact that for $\mathrm{HL}^{WI}$ the cut rule is admissible (see \ref{app_cut_hlwi}); it is also important to observe that the replacement of a hyperedge $e$ of a hypergraph $G$ by a hypergraph $H$ where $G$ and $H$ do not have isolated nodes, yields a hypergraph without isolated nodes as well. A nice consequence is that, while the cut rule cannot be eliminated in $\LC^\brac$ itself, it can be eliminated in $\mathrm{HL}^{WI}$ (and, particularly, in its fragment corresponding to $\LC^\brac$). However, this does not mean that we obtained a cut-free variant of $\LC^\brac$ for free: in a derivation of a sequent in $\mathrm{HL}^{WI}$, there can appear hypergraphs in antecedents that are not string ones (cycles, inverted strings etc.). 
	
	Regarding the formal language semantics of $A^\brac$, it is expected to be the following: $w(A^\brac)\eqdef w(A)^\neck\cup w(A)^{\neck\rev}$ (closure under both cyclic shifts and reversal). Syntactically, we can define the bracelet operation through the cyclic shift, the reversal, and the disjunction: $A^\brac \eqdef A^\neck \vee A^{\neck\rev}$, and then develop the Lambek calculus with $^\neck$, $^\rev$, and $\vee$ (this is out of scope of this paper). A curious open question is whether we can express $A^\brac$ through $\backslash$, $/$, $\cdot$, $^\neck$, and $^\rev$ (without the disjunction) at the level of semantics. 
	
	Finally, note that all the grammar results for $\LC^\neck$ from Section \ref{sec_grammars} can be directly transferred to $\LC^\brac$.
	\section{Conclusion}\label{sec_concl}
	The Lambek calculus with the cyclic shift operation $\LC^\neck$ introduced in this paper is in line with already considered and studied formalisms like $\LC^\rev$. The $^\neck$ operation does not seemingly have linguistic applications (though, as a joke, it can be pointed out that cycles appearing in the hypergraph calculus remind of the language of heptapods from the ``Arrival'' film released in 2016); however, this operation is interesting to study because cyclic shift of formal languages is a well-known operation studied in several works. In the paper, we have proved the cut elimination theorem, have studied recognizing power of grammars (in particular, we showed how to eliminate $A^\neck$ from antecedents), and have studied the possibility of embedding $\LC^\neck$ in the hypergraph Lambek calculus, which resulted in defining the bracelet operation $\LC^\brac$.
	
	The main open question is, of course, completeness w.r.t. the formal language semantics. Unforunately, a simple proof from \cite{Buszkowski82} does not work neither for $\LC^\neck$ nor for $\LC^\brac$, so we need to use more complex techniques. It would be also interesting to continue the study of recognizing power of grammars and to understand bounds of non-context-freeness of languages generated by $\LC^\neck$ better.
	
\bibliographystyle{plain}
\bibliography{NCL_Cyclic_shift_Preprint}

	\newpage
	\appendixpage
	\appendixtocoff
	\appendix
	
	\section{Proof of Lemma \ref{lemma_perm_simple}}\label{app_proof_lemma_perm_simple}
		\begin{proof}
			Induction on $n$. The base case is trivial. To prove the induction step, consider the derivation of $p_1/q_1,\dotsc, p_{n-1}/q_{n-1},p_n \to s$. Without loss of generality its two last steps can be represented as follows:
			$$
			\infer[(\CS)]{
				p_1/q_1,\dotsc, p_{n-1}/q_{n-1},p_n \to s
			}
			{
				\infer[(/\to)]{
					p_{i+1}/q_{i+1},\dotsc, p_{n-1}/q_{n-1},p_n, p_1/q_1,\dotsc, p_i/q_i \to s
				}
				{
					p_{i+1}/q_{i+1},\dotsc, p_{j-1}/q_{j-1}, p_j \to s
					&
					p_{j+1}/q_{j+1}, \dotsc, p_{n-1}/q_{n-1},p_n, p_1/q_1,\dotsc, p_i/q_i \to q_j
				}
			}
			$$
			Here $j>i\ge 0$ (if $i=0$, then the $(\CS)$ rule is not applied). We exploit the fact that each sequent within the derivation must contain exactly one type of the form $p_k$ in an antecedent (the proof is again by induction), and also the fact that $\Pi,\Psi\to A$ is derivable if and only if $\Psi,\Pi \to A$ is derivable. 
			
			Now, we apply the induction hypothesis and obtain a permutation $\sigma_1$ on $\{i+1,\dotsc,j-1\}$ and a permutation $\sigma_2$ on $\{1,\dotsc, i,j+1,\dotsc,n-1\}$ such that the conclusion of the lemma holds for the two abovestanding sequents in the fragment of the derivation. In order not to be drown in notations, let us denote by $\overline{p_{i+1}/q_{i+1},\dotsc, p_{j-1}/q_{j-1}}, p_j \to s$ the result of permuting types in the first sequent according to $\sigma_1$, and let us denote by $\underline{p_1/q_1,\dotsc, p_i/q_i,p_{j+1}/q_{j+1}, \dotsc, p_{n-1}/q_{n-1}},p_n \to q_j$ the result of permuting types in the second sequent according to $\sigma_2$. Now we can combine these sequents as follows:
			$$
			\overline{p_{i+1}/q_{i+1},\dotsc, p_{j-1}/q_{j-1}}, p_j/q_j, \underline{p_1/q_1,\dotsc, p_i/q_i,p_{j+1}/q_{j+1}, \dotsc, p_{n-1}/q_{n-1}},p_n \to s
			$$
			This sequent is obtained by permuting types $p_i/q_i$, $i=1,\dotsc,n-1$ of the initial sequent, and it is not hard to check that it satisfies the conditions of the lemma. 
		\end{proof}
	\section{Proof of Theorem \ref{th_even_calculus}}\label{app_proof_th_even_calculus}
	\begin{lemma}\label{lemma_to_box}\leavevmode
		\begin{itemize}
			\item $\LC^\neck \vdash  A \to A^\Box$.
			\item If $\LC^\neck \vdash A \to B$, then $\LC^\neck \vdash A^\Box \to B^\Box$.
		\end{itemize}
	\end{lemma}
	
	\begin{lemma}\label{lemma_double_rules}
		Consider a modification of $\LC^\neck$ that includes the following rules:
		$$
		\infer[(/\to)_2]{\Gamma,\Pi,(D\backslash E)/A,\Psi,\Delta \to C}{\Gamma, E, \Delta \to C & \Pi \to D & \Psi \to A}
		\qquad
		\infer[(\to\cdot)_2]{\Pi,\Psi,\Xi \to (D \cdot E)\cdot B}{\Pi \to D & \Psi \to E & \Xi \to B}
		$$
		$$
		\infer[(\to/)_2]{\Pi \to (D \backslash E)/A}{D , \Pi , A \to E}
		\qquad
		\infer[(\to\cdot)_2]{\Gamma,(D \cdot E)\cdot B,\Delta \to C}{\Gamma,D, E,B,\Delta \to C}
		$$
		In this modification (say $\LC^\neck_2$), we additionally forbid 
		\begin{itemize}
			\item to apply the rule $(/\to)$, if $B$ is of the form $D\backslash E$ ($B$ here is that from Section \ref{sec_intr} where the rule $(/\to)$ was firstly introduced),
			\item to apply the rule $(\to\cdot)$, if $A$ is of the form $D\cdot E$ ($A$ here is that from Section \ref{sec_intr} where the rule $(\to\cdot)$ was firstly introduced),
			\item to apply the rule $(\to/)$, if $B$ is of the form $D\backslash E$ ($A$ here is that from Section \ref{sec_intr} where the rule $(\to/)$ was firstly introduced),
			\item to apply the rule $(\cdot\to)$, if $A$ is of the form $D\cdot E$ ($A$ here is that from Section \ref{sec_intr} where the rule $(\cdot\to)$ was firstly introduced).
		\end{itemize}. Then a sequent is derivable in $\LC^\neck$ if and only if it is derivable in $\LC^\neck_2$.
	\end{lemma}
	This lemma states that, if we have a type with two divisions in an antecedent, we can remodel a derivation in such a way that the rule applications of $(\backslash\to)$ and $(/\to)$, after which this type appesrs, are done consecutively. Analogously, if we have a type with two products in a succeedent, then we can consecutively apply $(\to\cdot)$ two times for it to appear. The proof is done by a straightforward remodelling of a derivation of a sequent (one of the applications of rules should be moved to the other one).
		\begin{proof}[Proof sketch.]
			It is obvious that every sequent derivable in $\LC^\neck_1$ is derivable in $\LC^\neck$ (new rules just represent some combination of rules $(\backslash\to)$, $(/\to)$, and $(\to \cdot)$). The converse statement is proved by induction on the length of a derivation of a sequent in $\LC^\neck$. It suffices to consider only cases when a forbidden rule is applied the first time, and to show how to eliminate it.
			
			\textbf{Case 1.} The last rule in the derivation is $(/\to)$, after which a type of the form $(D\backslash E)/A$ appears in the antecedent:
			$$
			\infer[(/\to)]{\Gamma, (D\backslash E)/A, \Delta^\prime, \Delta^{\prime\prime} \to C}
			{
				\infer[]{\Gamma, D\backslash E , \Delta^{\prime\prime} \to C}
				{
					\infer[]{\dotsc\dotsc}
					{
						\infer[(\backslash\to)]{
							\widetilde{\Gamma}, \widetilde{\Pi}, D\backslash E , \widetilde{\Delta} \to \widetilde{C}
						}
						{
							\widetilde{\Gamma}, E , \widetilde{\Delta} \to \widetilde{C}
							&
							\widetilde{\Pi} \to D
						}
					}
				}
				&
				\Delta^\prime \to A
			}
			$$
			This derivation is remodeled into the following one:
			$$
			\infer[]{\Gamma, \boxed{(D\backslash E)/A, \Delta^\prime}, \Delta^{\prime\prime} \to C}
			{
				\infer[]{\dotsc\dotsc}
				{
					\infer[(/\to)]{\widetilde{\Gamma}, \widetilde{\Pi}, \boxed{(D\backslash E)/A, \Delta^\prime}, \widetilde{\Delta} \to \widetilde{C}}
					{
						\infer[(\backslash\to)]{
							\widetilde{\Gamma}, \widetilde{\Pi}, D\backslash E , \widetilde{\Delta} \to \widetilde{C}
						}
						{
							\widetilde{\Gamma}, E , \widetilde{\Delta} \to \widetilde{C}
							&
							\widetilde{\Pi} \to D
						}
						&
						\Delta^\prime \to A
					}
				}
			}
			$$
			A part of the antecedent is boxed; this means only that we regard it as a whole, and treat it similarly to $D\backslash E$ in the initial derivation; hence, we repeat the part of the derivation represented by ``$\dotsc\dotsc$'' replacing $D\backslash E$ by $(D\backslash E)/A,\Delta^\prime$ everywhere. Finally, we replace the applications of $(\backslash\to)$ and $(/\to)$ as follows:
			$$
			\infer[(/\to)_2]{\widetilde{\Gamma}, \widetilde{\Pi}, (D\backslash E)/A, \Delta^\prime, \widetilde{\Delta} \to \widetilde{C}}
			{
				\widetilde{\Gamma}, E , \widetilde{\Delta} \to \widetilde{C}
				&
				\widetilde{\Pi} \to D
				&
				\Delta^\prime \to A
			}
			$$
			
			\textbf{Case 2.} The last rule is $(\to\cdot)$, after which a type of the for $(D \cdot E)\cdot B$ appears. This case is proved similarly; we remodel the derivation as follows:
			$$
			\infer[(\to\cdot)]{\Gamma,\Delta \to (D\cdot E)\cdot B}
			{
				\infer[]{\Gamma \to D \cdot E}
				{
					\infer[]{\dotsc\dotsc}
					{
						\infer[(\to\cdot)]{
							\Gamma^\prime, \Gamma^{\prime\prime} \to D\cdot E
						}
						{
							\Gamma^\prime \to D
							&
							\Gamma^{\prime\prime} \to E
						}
					}
				}
				& 
				\Delta \to B
			}
			\quad\rightsquigarrow\quad
			\infer[]{\Gamma, \Delta \to (D \cdot E) \cdot B}
			{
				\infer[]{\dotsc\dotsc}
				{
					\infer[(\to\cdot )_2]{\Gamma^\prime, \Gamma^{\prime\prime},\Delta \to (D\cdot E)\cdot B}
					{
						\Gamma^\prime \to D
						&
						\Gamma^{\prime\prime} \to E
						&
						\Delta \to B
					}
				}
			}
			$$
			The part of the derivation denoted by ``$\dotsc\dotsc$'' is simply repeated. This can be done correctly because this part can consist only of applications of $(\backslash\to)$, $(/\to)$, and $(\cdot\to)$, and it acts only within $\Gamma$ not affecting the succeedent. 
			
			\textbf{Case 3.} If the last rule in the derivation is $(\to/)$, after which a type of the form $(D\backslash E)/A$ appears in the succeedent, then we remodel the derivation as follows:
			$$
			\infer[(\to/)]{\Pi \to (D\backslash E)/A}
			{
				\infer[]{\Pi,A \to D\backslash E}
				{
					\infer[]{\dotsc\dotsc}
					{
						\infer[(\to\backslash)]{
							\Pi^\prime \to D\backslash E
						}
						{
							D, \Pi^\prime \to E
						}
					}
				}
				&
				\Delta^\prime \to A
			}
			\qquad 
			\rightsquigarrow
			\qquad
			\infer[]{\Pi \to (D\backslash E)/A}
			{
				\infer[]{E, \Pi, A \to E}
				{
					\infer[(/\to)]{\dotsc\dotsc}
					{
						D, \Pi^\prime \to E
					}
				}
			}
			$$
			
			\textbf{Case 4.} If the last rule is $(\cdot\to)$, after which a type of the form $(D \cdot E)\cdot B$ appears in the antecedent,then we remodel the derivation as follows:
			$$
			\infer[(\cdot\to)]{\Gamma,(D\cdot E)\cdot B,\Delta \to C}
			{
				\infer[]{\Gamma,D\cdot E, B,\Delta \to C}
				{
					\infer[]{\dotsc\dotsc}
					{
						\infer[(\cdot\to)]{
							\widetilde{\Gamma},D\cdot E, \widetilde{\Delta} \to \widetilde{C}
						}
						{
							\widetilde{\Gamma},D, E, \widetilde{\Delta} \to \widetilde{C}
						}
					}
				}
			}
			\quad\rightsquigarrow\quad
			\infer[]{\Gamma,(D\cdot E)\cdot B,\Delta \to C}
			{
				\infer[]{\Gamma,D, E, B,\Delta \to C}
				{
					\infer[]{\dotsc\dotsc}
					{
						\widetilde{\Gamma},D, E, \widetilde{\Delta} \to \widetilde{C}
					}
				}
			}
			$$
		\end{proof}
	
	\begin{proof}[Proof of Theorem \ref{th_even_calculus}.]
		Firstly, let us prove two statements:
		\begin{enumerate}
			\item For $A_1,\dotsc,A_n,B$ from $Tp^\neck_{\le N}$ $\LC^\neck \vdash A_1, \dotsc, A_n\to B$ implies $\LC^\neck \vdash o_N(A_1), \dotsc, o_N(A_n)\to e_N(B)$.
			\item $\LC^\neck \vdash A \to B^\neck$ for $A \in Tp^\neck_{\le k}$ and $B\in Tp^\neck_{\le N}$ where $k \le N$ implies $\LC^\neck \vdash o_N(A) \to e_N(B)^{(\Box \neck)^k}$.
		\end{enumerate}. These statements are proved together by induction on the size of a derivation of $A_1, \dotsc, A_n\to B$ or $A \to B^\neck$. The base case is trivial, if we restrict the axiom $(Ax)$ to $p \to p$ for $p$ being primitive (this does not change the set of sequents derivable in $\LC^\neck$). The induction step depends on the last rule applied in the derivation. It is not hard to notice that, if the last rule belongs to $\LC$, then we can simply apply the induction hypothesis for premises, and apply the same rule obtaining a sequent of interest. For instance, let the last rule in the derivation of $A_1, \dotsc, A_n\to B$ be $(\to \cdot)$; then $B=B_1\cdot B_2$, and
		$$
		\infer[(\to \cdot)]{A_1, \dotsc, A_n \to B_1\cdot B_2}{A_1, \dotsc, A_i \to B_1 & A_{i+1}, \dotsc, A_n \to B_2}
		$$
		Then we apply the induction hypothesis to premises and construct the following derivation:
		$$
		\infer[(\to \cdot)]{o_N(A_1), \dotsc, o_N(A_n) \to e_N(B_1)\cdot e_N(B_2)}{o_N(A_1), \dotsc, o_N(A_i) \to e_N(B_1) & o_N(A_{i+1}), \dotsc, o_N(A_n) \to e_N(B_2)}
		$$
		Finally, notice that $e_N(B_1)\cdot e_N(B_2)=e_N(B_1\cdot B_2)$. Other rules of $\LC$ are dealt with similarly.
		
		Let the last rule applied in the derivation of $A_1, \dotsc, A_n\to B$ be $(\to^\neck)$. Then $B=C^\neck$, and
		$$
		\infer[(\to^\neck)]{A_1,\dotsc,A_n \to C^\neck}{A_1,\dotsc,A_n \to C}
		$$
		We apply the induction hypothesis for the premise and to the following:
		$$
		\infer[(\to^\neck)]{o_N(A_1),\dotsc, o_N(A_n) \to e_N(C)}{o_N(A_1),\dotsc, o_N(A_n) \to e_N(C) & e_N(C) \to (e_N(C))^{(\Box\neck)^N}}
		$$
		$e_N(C) \to (e_N(C))^{\Box\neck}$ is derivable using the rule $(\to\neck)$ and Lemma \ref{lemma_to_box}; we can repeat this $N$ times and obtain $e_N(C) \to (e_N(C))^{(\Box\neck)^N}$. Finally, notice that $(e_N(C))^{(\Box\neck)^N}=e_N(C^\neck)=e_N(B)$.
		
		Let the last rule applied be $(\neck)$; then $B=C^\neck$, and $e_N(B)=(e_N(C))^{(\Box \neck)^N}$. Hence the last operation in $e_N(B)$ is $^\neck$, and we can apply $(\neck)$.
		
		If the last rule applied to $A_1, \dotsc, A_n\to B$ is $(^\neck\to)$, then this is a particular case of the second statement since $A_1=A$ belongs to $Tp^\neck_{\le N}$.
		
		Now we turn to the secons statement. As earlier, if the last rule applied in a derivation of $A \to B^\neck$ is that from $\LC$, then we can simply repeat it. Rules $(\to^\neck)$ and $(\neck)$ are treated as above. Let the last rule applied be $(^\neck\to)$. Then $A=C^\neck$, and the rule is of the form
		$$
		\infer[(^\neck\to )]{C^\neck \to B^\neck}{C \to B^\neck}
		$$
		The size of $C$ is one less than that of $A$, hence $C \in Tp^\neck_{\le k-1}$. we apply the induction hypothesis and derive $o_N(C) \to (e_N(B))^{(\Box\neck)^{k-1}}$. Then we apply Lemma \ref{lemma_to_box} and do the following:
		$$
		\infer[(\to^\neck)]{
			o_N(C)^\Box \to (e_N(B))^{(\Box\neck)^{k}}
		}{
			\infer[]{
				o_N(C)^\Box \to (e_N(B))^{(\Box\neck)^{k-1}\Box}
			}{
			o_N(C) \to (e_N(B))^{(\Box\neck)^{k-1}}
			}
		}
		$$
		Finally, notice that $o_N(C)^\Box = o_N(C^\neck)$. This completes the proof the two statements formulated in the beginning of the proof, and consequently the proof of the ``only if'' part of the theorem.
		
		Let us prove the converse statement: for $A_1,\dotsc,A_n,B$ from $Tp^\neck_{\le N}$ $\LC^\neck \vdash o_N(A_1), \dotsc, o_N(A_n)\to e_N(B)$ implies $\LC^\neck \vdash A_1, \dotsc, A_n\to B$. First of all, let $A_i^\prime$ be obtained from $o_N(A_i)$ by replacing each $^\Box$ connective by $^\neck$, and let $B^\prime$ obtained from $e_N(B)$ similarly. Then $\LC^\neck \vdash A^\prime_1, \dotsc, A^\prime_n\to B^\prime$ implies $\LC^\neck \vdash A_1, \dotsc, A_n\to B$ because the difference between types in these sequents is the number of $^\neck$ operations; however, $A^\neck \leftrightarrow A^{\neck\neck}$, so we can replace one cyclic shift operation by an arbitrary positive number of them. Therefore, it suffices to prove that $\LC^\neck \vdash o_N(A_1), \dotsc, o_N(A_n)\to e_N(B)$ implies $\LC^\neck \vdash A^\prime_1, \dotsc, A^\prime_n\to B^\prime$.
		
		Let us denote by $Tp^{\neck,\Box}$ the set of types built from primitive ones (not including $l,r$) using operations $\backslash$, $\cdot$, $/$, $\neck$, and $\Box$. 
		
		Consider a cut-free derivation of the sequent $o_N(A_1), \dotsc, o_N(A_n)\to e_N(B)$ in the calculus $\LC^\neck_2$ (it is defined in Lemma \ref{lemma_double_rules}). We claim that, if we replace all $^\Box$ in this derivation by $^\neck$, then it remains correct. Firstly, let us show that only three kinds of sequents may appear within this derivation:
		\begin{enumerate}
			\item A sequent of the form $C_1, \dotsc, C_m\to D$ where $C_i, D \in Tp^{\neck,\Box}$;
			\item A sequent of the form $l, C_1, \dotsc, C_m, r \to (l \cdot D)\cdot r$ where $C_i, D \in Tp^{\neck,\Box}$;
			\item A sequent of the form $(l \cdot C)\cdot r \to (l \cdot D)\cdot r$ where $C, D \in Tp^{\neck,\Box}$;
			\item $l \to l$ and $r \to r$.
		\end{enumerate}
		This is proved by induction on the length of a derivation. The axiom case belongs either to kind 1 or to kind 4. To prove the induction step, we need the following
		\begin{lemma}
			If there is a sequent of the form $\Pi \to l$ ($\Pi \to r$) in the derivation of $o_N(A_1), \dotsc, o_N(A_n)\to e_N(B)$, then $\Pi=l$ ($\Pi = r$ resp.).
		\end{lemma}
		To prove this lemmma, it suffices to notice that no rule can be applied to the sequent $l \to l$ in such a way that the antecedent remains the same, and types in the antecedent are subtypes of types from $o_N(A_1), \dotsc, o_N(A_n)\to e_N(B)$.
		
		Now, we proceed with proving the induction step. We consider the last rule applied in the derivation; applying the induction hypothesis, we can always assume that the premises of the last rule application are of one of three kinds described above. 
		\begin{enumerate}
			\item The last rule applied is $(/\to)$:
			$$
			\infer[(/\to)]{\Gamma, B / A, \Pi, \Delta \to C}{\Gamma, B, \Delta \to C & \Pi \to A }
			$$
			The following cases are possible:
			\begin{enumerate}
				\item $\Gamma, B, \Delta \to C$ and $\Pi \to A$ are of kind 1; then $\Gamma, B / A, \Pi, \Delta \to C$ is of kind 1 as well.
				\item $\Gamma, B, \Delta \to C$ is of kind 2, and $\Pi \to A$ is of kind 1. Then $l$ is necessarily in $\Gamma$, and $r$ is necessarily in $\Delta$; the rule application does not affect them, hence we obtain a sequent of kind 2.
			\end{enumerate}
			\item Simiarly, rules $(\backslash \to)$, $(\to /)$, $(\to \backslash)$, $(\cdot \to)$, $(\to \cdot)$, $(\to ^\neck)$, and $(\neck)$ are considered.
			\item If rules $(/ \to)_2$, $(\to /)_2$, $(\cdot \to)_2$, or $(\to \cdot)_2$, are applied in such a way that they do not involve $l$ and $r$, then such cases can be considered similarly (moreover, for these cases we can decompose each of such rules into two ones of $\LC$).
			\item The rule $(/ \to)_2$ involves $l$ and $r$. Then, the rule application must be of the form
			$$
			\infer[(/\to)_2]{\Gamma, l, (l \backslash B)/r, r, \Delta \to C}{\Gamma, B, \Delta \to C & l \to l & r \to r }
			$$
			Since these sequents occur within the derivation of $o_N(A_1), \dotsc, o_N(A_n)\to e_N(B)$, the type $(l \backslash B)/r$ is a subtype of $o_N(A_i)$ for some $i$ or of $e_N(B)$. Consequently, $B$ must be of the form $(l \cdot D) \cdot r$, and $C$ must be of the form $(l \cdot E) \cdot r$. The sequent $\Gamma, (l \cdot D)\cdot r, \Delta \to C$ can be only of kind 3, if $\Gamma=\Delta=\varepsilon$; in such a case the derivation is the following:
			$$
			\infer[(/\to)_2]{l, D^\Box, r \to (l \cdot E) \cdot r}{(l \cdot D) \cdot r \to (l \cdot E) \cdot r & l \to l & r \to r }
			$$
			\item The rule $(\cdot \to)_2$ involves $l$ and $r$. Then, the rule application must be of the form
			$$
			\infer[(\cdot \to)_2]{(l \cdot A) \cdot r \to (l \cdot B) \cdot r}{l, A, r \to (l \cdot B) \cdot r}
			$$
			Here the premise is necessarily of kind 2, and the conclusion is of kind 3. Notice that it is necessary that there is only one type except for $l$ and $r$ in the premise since otherwise we cannot apply the rule $(\cdot \to)_2$ to both $l$ and $r$ (this is why we needed to introduce this rule).
			\item The rule $(\to /)_2$ involves $l$ and $r$. Then, the rule application must be of the form
			$$
			\infer[(\to/)_2]{C_1, \dotsc, C_m \to D^\Box}{l, C_1, \dotsc, C_m, r \to (l \cdot D) \cdot r}
			$$
			Hence, it transforms a sequent of kind 2 into a sequent of kind 1.
			\item The rule $(\to \cdot )_2$ involves $l$ and $r$. Then, the rule application must be of the form
			$$
			\infer[(\to \cdot )_2]{l, \Pi, r \to (l \cdot A)\cdot r}{\Pi \to A & l \to l & r \to r}
			$$
			It tranforms a sequent of kind 1 into a sequent of kind 2.
		\end{enumerate}
		Finally, notice that the rule $(^\neck \to )$ cannot be applied (because the types are designed in such a way). 
		
		Now, let us transform the derivation of $o_N(A_1), \dotsc, o_N(A_n)\to e_N(B)$ as follows:
		\begin{itemize}
			\item Rule applications transforming sequents of kind 1 into a sequent of type 1 remain the same.
			\item The rule application of $(\to \cdot )_2$ transforming a sequent of kind 1 into a sequent of kind 2 is replaced by the rule application of $(\to ^\neck)$:
			$$
			\infer[(\to \cdot )_2]{l, \Pi, r \to (l \cdot A)\cdot r}{\Pi \to A & l \to l & r \to r}
			\qquad\rightsquigarrow\qquad
			\infer[(\to ^\neck)]{\Pi \to A^\neck}{\Pi \to A}
			$$
			\item Applications of rules to sequents of kind 2 that do not affect $l$ and $r$ are remodeled as follows:
			$$
			\infer[]{
				l, \Pi^\prime, r \to (l \cdot A)\cdot r
			}{
				l, \Pi, r \to (l \cdot A)\cdot r & \dotsc
			}
			\qquad\rightsquigarrow\qquad
			\infer[]{
				\Pi^\prime \to A^\neck
			}{
				\Pi \to A^\neck & \dotsc
			}
			$$
			``$\dotsc$'' may be empty if the rule $(\cdot \to)$ is applied within $\Pi$, or be a sequent of kind 1 if the rule $(/ \to)$ or $(\backslash \to)$ is applied.
			\item The rule application of $(\cdot \to)_2$ transforming a sequent of kind 2 into a sequent of kind 3 is replaced by the rule application of $(^\neck \to)$:
			$$
			\infer[(\cdot \to)_2]{(l \cdot A) \cdot r \to (l \cdot B) \cdot r}{l, A, r \to (l \cdot B) \cdot r}
			\qquad\rightsquigarrow\qquad
			\infer[(^\neck \to )]{A^\neck \to B^\neck}{A \to B^\neck}
			$$
			\item The rule application of $(/ \to)_2$ transforming a sequent of kind 3 into a sequent of kind 2 is removed (it works together with the rule $(\cdot \to)_2$ described above, so they can be considered jointly as the rule application of $(^\neck \to)$).
			\item The rule application of $(\to /)_2$ transforming a sequent of kind 2 into a sequent of kind 1 is removed (it works together with the rule $(\to \cdot )_2$, so they can be considered jointly as the rule application of $(\to ^\neck)$).
		\end{itemize}
		Therefore, we remodeled a derivation and replaced each $^\Box$ by $^\neck$. The new derivation is correct in $\LC^\neck$, and the resulting sequent is $A_1^\prime, \dotsc,A_n^\prime \to B^\prime$; hence, it is derivable as we aimed to show.
	\end{proof}

	\section{Cut Elimination Theorem for $\mathrm{HL}^{WI}$}\label{app_cut_hlwi}
	\begin{definition}
		Size $|T|$ of a type $T$ is defined inductively as follows:
		\begin{enumerate}
			\item $|p|\eqdef 1$ for $p\in Pr$;
			\item If $T=N\div D$ and $E_D=\{d_0,\dots,d_k\}$ with $lab_D(d_0)$ being equal to \$, then $|T|\eqdef |N|+|lab_D(d_1)|+\dots+|lab_D(d_k)|+1$;
			\item If $T=\times(M)$ and $E_M=\{m_1,\dots,m_k\}$, then $|T|\eqdef |lab_M(m_1)|+\dots+|lab_M(m_k)|+1$.
		\end{enumerate}
	\end{definition}
	\begin{theorem}\label{th_cut_hlwi}
		If $\mathrm{HL}^{WI}\vdash H\to A$ and $\mathrm{HL}^{WI}\vdash G\to B$, then $\mathrm{HL}^{WI}\vdash G[e_0/H]\to B$ where $e_0\in E_G$ and $lab(e_0)=A$.
	\end{theorem}
	\begin{proof}[Proof (of Theorem \ref{th_cut_hlwi}).]
		The proof is by induction on $|\times(H)|+|A|+|\times(G)|+|B|$.
		
		\textbf{Case 1.} $H\to A$ is an axiom $p^\bullet\to p$. Then $G[e_0/H]=G$, so the replacement changes nothing. 
		
		\textbf{Case 2.} $G\to B$ is an axiom $p^\bullet\to p$. Then $A=lab(e_0)=B$, and $G[e_0/H]\to B = H\to A$, so the conclusion coincides with one of the premises.
		
		Let us further call the distinguished type $N\div D$ in rules $(\div\to)$ and $(\to\div)$, and the distinguished type $\times(M)$ in rules $(\times\to)$ and $(\to\times)$ (we mean those from definitions in Section \ref{ssec_def_HL}) the \emph{major type of the rule}.
		
		\textbf{Case 3.} In $H\to A$, the type $A$ is not the major type of the last rule applied. There are two subcases depending on the type of this rule. 
		
		\textbf{Case 3a.} $(\div\to)$:
		$$
		\infer[(\mathrm{cut})]{G[e_0/H]\to B}{
			\infer[(\div\to)]{H\to A}{
				K\to A & H_1\to T_1 & \dots & H_k\to T_k
			}
			&
			G\to B
		}
		$$
		Here $H$ is obtained from $K$ by replacements using $H_1,\dots,H_k$, as the rule $(\div\to)$ prescribes. Note that we omit some details of rule applications that are not essential here (but their role can be understood from the general structure of the rule).
		\\
		This derivation is transformed as follows:
		$$
		\infer[(\div\to)]{G[e_0/H]\to B}{
			\infer[(\mathrm{cut})]{G[e_0/K]\to B}{
				K\to A & G\to B
			}
			&
			H_1\to T_1 & \dots & H_k\to T_k
		}
		$$
		Now we apply the induction hypothesis to the premises and obtain a cut-free derivation for $G[e_0/H]\to B$. Further the induction hypothesis will be applied in a similar way to the premises appearing in the new derivation process. Sometimes the induction hypothesis will be applied several times (from top to bottom, see Cases 5 and 6); however, this will be always legal.
		
		\textbf{Case 3b} $(\times\to)$. Let $f_0\in E_H$ be labeled by a type $\times(K)$, which apears at the last step of a derivation. Then the remodelling is as follows:
		$$
		\infer[(\mathrm{cut})]{G[e_0/H]\to B}{
			\infer[(\times\to)]{H\to A}{
				H[f_0/K]\to A
			}
			&
			G\to B
		}
		\quad\rightsquigarrow\quad
		\infer[(\times\to)]{G[e_0/H]\to B}{
			\infer[(\mathrm{cut})]{G[e_0/H[f_0/K]]\to A}{
				H[f_0/K]\to A & G\to B
			}
		}
		$$
		
		\textbf{Case 4.} The type $A$ labeling $e_0$ within $G$ is not the major type in the last rule in the derivation of $G\to B$. Then one repeats the last step of the derivation of $G\to B$ in $G[e_0/H]\to B$ considering $H$ to be an atomic structure acting as $e_0$. Formally, there are five subcases depending on the last rule applied in the derivation of $G\to B$:
		\begin{enumerate}
			\item $(\div\to)$ if one of the subgraphs $H_i$ contains $e_0$:
			$$
			\infer[(\mathrm{cut})]{G[e_0/H]\to B}{
				H\to A
				&
				\infer[(\div\to)]{G\to B}{
					K\to B & H_1\to T_1\;\dots & H_i\to T_i &\dots\; H_k\to T_k
				}
			}
			$$
			Let $H_i$ contain an edge $e_0$; then this derivation is remodeled as follows:
			$$
			\infer[(\div\to)]{G[e_0/H]\to B}{
				K\to B & H_1\to T_1\;\dots & \infer[(\mathrm{cut})]{H_i[e_0/H]\to T_i}{H\to A & H_i\to T_i} &\dots\; H_k\to T_k
			}
			$$
			
			\item $(\div\to)$ if $e_0$ is not contained in any $H_i$ (then $e_0$ belongs to $E_K$):
			$$
			\infer[(\mathrm{cut})]{G[e_0/H]\to B}{
				H\to A
				&
				\infer[(\div\to)]{G\to B}{
					K\to B & H_1\to T_1\; & \dots & \; H_k\to T_k
				}
			}
			$$
			\begin{center}
				$\downsquigarrow$
			\end{center}
			$$
			\infer[(\div\to)]{G[e_0/H]\to B}{
				\infer[(\mathrm{cut})]{K[e_0/H]\to B}
				{H\to A & K\to B} &
				H_1\to T_1\; & \dots & \; H_k\to T_k
			}
			$$
			\item $(\times\to)$: 
			$$
			\infer[(\mathrm{cut})]{G[e_0/H]\to B}{
				H\to A
				&
				\infer[(\times\to)]{G\to B}{G[f_0/K]\to B}
			}
			$$
			Here $f_0\in E_G$ is labeled by $\times(K)$ (and $f_0\ne e_0$, because $A$ is not major). The remodelling is as follows:
			$$
			\infer[(\to\times)]{G[e_0/H]\to B}{
				\infer[(\mathrm{cut})]{G[f_0/K][e_0/H]\to B}{
					H\to A
					&
					G[f_0/K]\to B
				}
			}
			$$
			\item $(\to\div)$:
			$$
			\infer[(\mathrm{cut})]{G[e_0/H]\to B}
			{
				H\to A
				&
				\infer[(\to\div)]{G\to N\div D}{D[d_0/G]\to N}
			}
			$$
			Here $d_0\in E_D$ is labeled by \$. Then
			$$
			\infer[(\to\div)]{G[e_0/H]\to N\div D}
			{
				\infer[(\mathrm{cut})]{D[d_0/G[e_0/H]]\to N}{
					H\to A
					&
					D[d_0/G]\to N
				}
			}
			$$
			Here we use the associativity property: $D[d_0/G[e_0/H]]=D[d_0/G][e_0/H]$.
			\item $(\to\times)$: 
			$$
			\infer[(\mathrm{cut})]{G[e_0/H]\to \times(M)}{
				H\to A
				&
				\infer[(\to\times)]{G\to \times(M)}{
					H_1\to T_1\;\dots & H_i\to T_i &\dots\; H_k\to T_k 
				}
			}
			$$
			Here $G$ is composed of copies of $H_1,\dots,H_k$ by means of $M$. Since $e_0\in E_G$, there is such a graph $H_i$ that $e_0\in E_{H_i}$. Then we can remodel this derivation as follows:
			$$
			\infer[(\to\times)]{G[e_0/H]\to \times(M)}{
				H_1\to T_1\;\dots & \infer[(\mathrm{cut})]{H_i[e_0/H]\to T_i}
				{H\to A & H_i\to T_i} &\dots\; H_k\to T_k
			}
			$$
		\end{enumerate}
		
		\textbf{Case 5.} $A=\times(M)$ is major in both $H\to A$ and $G\to B$.
		$$
		\infer[(\mathrm{cut})]{G[e_0/H]\to B}
		{
			\infer[(\to\times)]{H\to\times(M)}{H_1\to T_1 & \dots & H_k\to T_k}
			&
			\infer[(\times\to)]{G\to B}{G[e_0/M]\to B}
		}
		$$
		Let us denote $E_M=\{e_1,\dots,e_k\}$ and $lab_M(e_i)=T_i$. Note that $M$ is a subgraph of $K\eqdef G[e_0/M]$, in particular $E_M\subseteq E_K$. Now we are ready to remodel this derivation as follows:
		$$
		\infer[(\mathrm{cut})]{K[e_1/H_1]\dots[e_k/H_k]\to B}
		{
			H_k\to T_k
			&
			\infer[]{K[e_1/H_1]\dots[e_{k-1}/H_{k-1}]\to B}
			{
				\infer[(\mathrm{cut})]{\dots}
				{
					H_2\to T_2
					&
					\infer[(\mathrm{cut})]{K[e_1/H_1][e_2/H_2]\to B}
					{
						\infer[(\mathrm{cut})]{K[e_1/H_1]\to B}
						{
							H_1\to T_1
							&
							K\to B
						}
					}
				}
			}
		}
		$$
		Finally, note that $K[e_1/H_1]\dots[e_k/H_k]=G[e_0/H]$. The induction hypothesis applied several times from top to bottom of this new derivation completes the proof.
		
		\textbf{Case 6.} $A=N\div D$ is major in both $H\to A$ and $G\to B$.
		$$
		\infer[(\mathrm{cut})]{G[e_0/H]\to B}
		{
			\infer[(\to\div)]{H\to N\div D}{D[d_0/H]\to N}
			&
			\infer[(\div\to)]{G\to B}{L\to B & H_1\to T_1 & \dots & H_k\to T_k}
		}
		$$
		Here $d_0\in E_D$ is labeled by \$. We denote edges in $E_D$ except for $d_0$ as $e_1,\dots,e_k$; let $lab_D(e_i)=T_i$ (from above). Note that $e_1,\dots,e_k$ can be considered as edges of $K\eqdef D[d_0/H]$ as well. Observe that $L$ has to contain an edge labeled by $N$ that participates in $(\div\to)$; denote this edge by $\widetilde{e}_0$. Then the following remodelling is done:
		$$
		\infer[(\mathrm{cut})]{L[\widetilde{e}_0/K][e_1/H_1]\dots[e_k/H_k]\to B}
		{
			H_1\to T_1
			&
			\infer[]{L[\widetilde{e}_0/K][e_1/H_1]\dots[e_{k-1}/H_{k-1}]\to B}
			{
				\infer[(\mathrm{cut})]{\dots}
				{
					H_2\to T_2
					&
					\infer[(\mathrm{cut})]{L[\widetilde{e}_0/K][e_1/H_1][e_2/H_2]\to B}
					{
						\infer[(\mathrm{cut})]{L[\widetilde{e}_0/K][e_1/H_1]\to B}
						{
							H_1\to T_1
							&
							\infer[(\mathrm{cut})]{L[\widetilde{e}_0/K]\to B}
							{
								K\to N
								&
								L\to B
							}
						}
					}
				}
			}
		}
		$$
		As a final note, we observe that $L[\widetilde{e}_0/K][e_1/H_1]\dots[e_k/H_k]=G[e_0/H]$. This completes the proof.
		
		We would also like to bring your attention to the fact that new hypergraphs occuring in remodelled derivations are obtained by one or several replacements where all hypergraphs involved in it are without isolated nodes; hence, the resulting hypergraph is without isolated nodes as well. This was exploited in the proof, and this reflects the special nature of $\mathrm{HL}^{WI}$.
	\end{proof}
\end{document}